\documentclass[a4paper,11pt]{article}
\usepackage[utf8]{inputenc}
\usepackage{amsmath}
\usepackage{amssymb}
\usepackage{amsthm}
\usepackage{fullpage}
\usepackage{graphicx}
\usepackage{mathbbol,mathtools}
\usepackage{color}
\usepackage{times,enumerate}

\usepackage{booktabs}
\usepackage{tikz}
\newcommand{\mysqr}[1]{rectangle +(0.5,0.5) +(0.25,0.25) node{#1}}

\usepackage[colorlinks = true,
            linkcolor = blue,
            urlcolor  = blue,
            citecolor = blue,
            anchorcolor = blue]{hyperref}
\usepackage{booktabs,xcolor,siunitx,wrapfig}
\usepackage[skip=0pt]{caption}
\usepackage{mathrsfs}
\usepackage[font={color=black,it},figurename=Figure,labelfont={it}]{caption}

\def\le{\leqslant}
\def\ge{\geqslant}
\def\tr#1{\lfloor #1\rfloor}

\def\cl#1{\lceil #1\rceil}

\def\lpa#1{\bigl({#1}\bigr)}
\def\Lpa#1{\Bigl({#1}\Bigr)}

\def\llpa#1{\biggl({#1}\biggr)}
\def\pp{\mathbb{p}}
\def\dd#1{{\,\rm d}#1}
\def\with{\;\text{ with }\;}
\def\and{\;\text{and}\;}
\def\ve{\varepsilon}
\def\divides{\,|\,}

\newtheorem{thm}{Theorem}[section]
\newtheorem{prop}[thm]{Proposition}
\newtheorem{lem}[thm]{Lemma}
\newtheorem{coro}[thm]{Corollary}

\theoremstyle{remark}
\newtheorem{rmk}{Remark}

\DeclareRobustCommand{\Stirling}{\genfrac\{\}{0pt}{}}

\newcommand{\tdef}[1]{\emph{#1}}

\newcommand{\defeq}{:=}

\newcommand{\Res}{\mathrm{Res}}
\renewcommand{\Re}{\operatorname{Re}}

\graphicspath{{./hx-figs/}}

\makeatletter
\let\@fnsymbol\@arabic
\makeatother
    
\begin{document}

\title{Phase transitions from $\exp(n^{1/2})$ to $\exp(n^{2/3})$ in
the asymptotics of banded plane partitions\thanks{This work was
partially supported by a joint FWF-MOST project under the Grants I2309
(FWF) and 104-2923-M-009-006-MY3 (MOST).}}

\author{Wenjie Fang\\
    {\normalsize \emph{LIGM, Univ.~Gustave Eiffel, CNRS, ESIEE Paris,
     F-77454, Marne-la-Vall\'ee, France}}\\
    \quad\\
Hsien-Kuei Hwang\thanks{Also partially supported by an Investigator 
Award from Academia Sinica under the Grant AS-IA-104-M03 and a joint 
ANR-MOST project under the Grants MOST 105-2923-E-001-MY4.}\\
    {\normalsize \emph{Institute of Statistical Science, 
    Academia Sinica,
    Taipei 11529, Taiwan}}\\
    \quad\\
Mihyun Kang\\
    {\normalsize \emph{Institute of Discrete Mathematics,
    Graz University of Technology,
    8010 Graz,
    Austria}}}
\date{\today}

\maketitle
\begin{abstract}

We examine the asymptotics of a class of banded plane partitions
under a varying bandwidth parameter $m$, and clarify the transitional
behavior for large size $n$ and increasing $m=m(n)$ to be from $c_1
n^{-1} \exp\lpa{c_2 n^{1/2}}$ to $c_3 n^{-49/72} \exp\lpa{c_4
n^{2/3}+c_5 n^{1/3}}$ for some explicit coefficients $c_1,\dots,c_5$.
The method of proof, which is a unified saddle-point analysis for all
phases, is general and can be extended to other classes of plane
partitions.

\end{abstract}
 
\section{Introduction}

Asymptotics of partition-related generating functions with the unit
circle as the natural boundary has been the subject of study since
Hardy and Ramanujan's 1918 epoch-making paper \cite{Hardy1918}.
In particular, it is known that the number of partitions of $n$ into
positive integers is asymptotic to
\begin{align}\label{eq:pn-as}
    p_n := [z^n]\prod_{k \ge 1} \frac1{1-z^k}
    \sim c n^{-1}e^{\beta n^{1/2}},
    \with (c,\beta) = \llpa{\frac{1}{4 \sqrt{3}},
    \frac{\sqrt{2}\, \pi}{\sqrt{3}}},
\end{align}
(see \cite{Andrews1976,Hardy1918} or 
\cite[\href{https://oeis.org/A000041}{A000041}]{oeis2019}), and that 
of \emph{plane partitions} of $n$ satisfies 
\begin{align}\label{eq:pp}
    \pp_n^{} = [z^n]\prod_{k\ge1}\frac1{\lpa{1-z^k}^k}
    \sim c n^{-25/36} e^{\beta n^{2/3}},
    \with (c,\beta) = \llpa{\frac{\zeta(3)^{7/36}
    e^{-\zeta'(-1)}}{2^{11/36} \sqrt{3\pi}},
    \frac{3\zeta(3)^{1/3}}{2^{2/3}}},
\end{align}
(see \cite{Andrews1976,Wright1931} or
\cite[\href{https://oeis.org/A000219}{A000219}]{oeis2019}). Here the
symbol $[z^n]f(z)$ denotes the coefficient of $z^n$ in the Taylor
expansion of $f$ and $\zeta(s)$ the Riemann zeta function
\cite{Apostol1976,Whittaker1996}. \emph{Throughout this paper, the
values of the generic (or local) symbols $c,\beta$ or $c_j, \beta_j$
may differ from one occurrence to the other, and will always be
locally specified.}

The increase of the sub-exponential (or stretched exponential) term
from $e^{\beta n^{1/2}}$ in the case of ordinary partitions to
$e^{\beta n^{2/3}}$ in the case of plane partitions is noticeable,
and marks the essential difference in the respective asymptotic
enumeration. As integer partitions are also encountered in
statistical physics, astronomy, and other engineering applications,
one naturally wonders if there is a tractable combinatorial model
that interpolates between the two different orders $e^{n^{1/2}}$ and
$e^{n^{2/3}}$ when some structural parameter varies. This paper aims
to address this aspect of partition asymptotics and examines in
detail a class of plane partitions with a natural notion of bandwidth
$m$ whose variation yields a model in which we can fully clarify the
transitional behavior from being of order $e^{\beta n^{1/2}}$ for
bounded $m$ to $e^{\beta n^{2/3}}$ when $m\gg n^{1/3}$, providing
more modeling flexibility of these partitions. Our study constitutes
the first asymptotic realization of such phase transitions in the
analytic theory of partitions. Readers are referred to 
\cite[Section~VII.10]{Flajolet2009} for an introduction to phase 
transitions in combinatorial structures.

Intuitively, if we impose a constraint to one or two of the 
dimensions of plane partitions, then by suitably varying the 
constraint, we can generate families of objects whose asymptotic
behaviors interpolate between $e^{n^{1/2}}$ and $e^{n^{2/3}}$. An 
initial attempt can be found, \textit{e.g.}, in \cite{Gordon1969},
where Gordon and Houten computed the asymptotic counting formula for
``$k$-rowed partitions'' whose nonzero parts decrease strictly along
each row of size $n$. However, they studied only the situations when 
$k$ is bounded and when $k\to\infty$, and do not consider how exactly 
the asymptotic behavior changes with respect to varying $k$ 
(depending on $n$). See Section~\ref{sec:m-rowed} for the phase 
transitions in plane partitions with a given number of rows. 

The plane partitions of $n\ge0$ may be viewed as a matrix with
nonincreasing entries along rows and columns and with the entry-sum
equal to $n$. The class of plane partitions we work on in this paper
is the \emph{double shifted plane partitions} studied by Han and
Xiong in \cite{Han2019} with an explicit notion of width, which for
simplicity will be referred to as the \emph{banded plane partitions}
(or \emph{BPPs}) in this paper. These are plane partitions arranged
on the \emph{stair-shaped region} $\mathbb{T}_{m} = \{
(i,j)\in\mathbb{N}^2\mid j\le i\le j+m-1\}$, $m\in\mathbb{Z}^+$,
where $\mathbb{N}=\mathbb{Z}^+\cup\{0\}$. Formally, a banded plane
partition of width $m$ is a function $f: \mathbb{T}_m \to \mathbb{N}$
with finite support such that, for any $(i,j) \in \mathbb{T}_m$, we
have $f(i,j) \ge f(i,j+1)$ when $(i,j+1) \in \mathbb{T}_m$, and
$f(i,j) \ge f(i+1,j)$ when $(i+1,j) \in \mathbb{T}_m$.
Figure~\ref{fig:bpp-example} illustrates two instances of BPPs.

\begin{figure}[!thbp]
\centering
\begin{tikzpicture}
\begin{scope}
\draw (0,0) \mysqr{$7$};
\draw (0.5,0) \mysqr{$7$};
\draw (1,0) \mysqr{$4$};
\draw (1.5,0) \mysqr{$2$};
\draw (0.5,0.5) \mysqr{$6$};
\draw (1,0.5) \mysqr{$4$};
\draw (1.5,0.5) \mysqr{$2$};
\draw (2,0.5) \mysqr{$2$};
\draw (1,1) \mysqr{$3$};
\draw (1.5,1) \mysqr{$1$};
\draw (2,1) \mysqr{$1$};
\draw (1.5,1.5) \mysqr{$1$};
\end{scope}
\begin{scope}[xshift=3cm]
\draw (0,0) \mysqr{7};
\draw (0.5,0) \mysqr{7};
\draw (1,0) \mysqr{4};
\draw (1.5,0) \mysqr{2};
\draw (0.5,0.5) \mysqr{6};
\draw (1,0.5) \mysqr{4};
\draw (1.5,0.5) \mysqr{2};
\draw (2,0.5) \mysqr{2};
\draw (1,1) \mysqr{3};
\draw (1.5,1) \mysqr{1};
\draw (2,1) \mysqr{1};
\draw (1.5,1.5) \mysqr{1};
\draw[line width=1.2] (0,0) -- (0,0.5) -- (0.5,0.5) -- (0.5,1) -- (1,1) -- (1,1.5) -- (1.5,1.5) -- (1.5,2) -- (2,2) -- (2,2.5);
\draw[line width=1.2] (2,0) -- (2,0.5) -- (2.5,0.5) -- (2.5,1) -- (3,1) -- (3,1.5) -- (3.5,1.5) -- (3.5,2) -- (4,2);
\draw[line width=1.2] (0,0) -- (2,0);
\draw[dashed, line width=1.2] (2,2.5) -- (2.5,3);
\draw[dashed, line width=1.2] (4,2) -- (4.5,2.5);
\end{scope}
\end{tikzpicture}
\begin{tikzpicture}
\begin{scope}
\draw (0,0) \mysqr{15};
\draw (0.5,0) \mysqr{10};
\draw (1,0) \mysqr{7};
\draw (0.5,0.5) \mysqr{8};
\end{scope}
\begin{scope}[xshift=2cm]
\draw (0,0) \mysqr{15};
\draw (0.5,0) \mysqr{10};
\draw (1,0) \mysqr{7};
\draw (0.5,0.5) \mysqr{8};
\draw[line width=1.2] (0,0) -- (0,0.5) -- (0.5,0.5) -- (0.5,1) -- (1,1) -- (1,1.5) -- (1.5,1.5) -- (1.5,2) -- (2,2) -- (2,2.5);
\draw[line width=1.2] (2,0) -- (2,0.5) -- (2.5,0.5) -- (2.5,1) -- (3,1) -- (3,1.5) -- (3.5,1.5) -- (3.5,2) -- (4,2);
\draw[line width=1.2] (0,0) -- (2,0);
\draw[dashed, line width=1.2] (2,2.5) -- (2.5,3);
\draw[dashed, line width=1.2] (4,2) -- (4.5,2.5);
\end{scope}
\end{tikzpicture}
\medskip
\caption{Two instances of banded plane partition of size $40$ and 
width $4$ (with and without the outer banded staircase).}
\label{fig:bpp-example}
\end{figure}
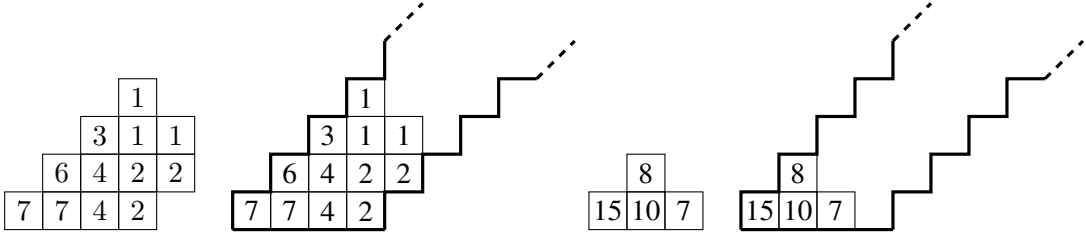

The \tdef{size} of a BPP is the sum $\sum_{(i,j) \in \mathbb{T}_m}
f(i,j)$. We denote by $G_{n,m}$ the number of BPPs of size $n$ and  
width $m$, \textit{i.e.}, BPPs that can fit in $\mathbb{T}_m$. A
closed-form expression for the generating function $G_m(z)\defeq
\sum_{n \ge 0} G_{n,m} z^n$ is given in \cite[Theorem~1.1]{Han2019}
as $G_m(z) = P(z) Q_m(z)$, where
\begin{equation} \label{eq:han-xiong}
\begin{split}
    P(z) &= \prod_{k \ge 1} \frac1{1-z^k},
    \;\text{ and } \;
	Q_m(z) = \prod_{k\ge0}\prod_{1\le h < j <m}
	\frac1{1-z^{2mk+h+j}}.
\end{split}
\end{equation}
In particular, 
\begin{align*}
    Q_3(z) &= \prod_{k\ge0}\frac1{1-z^{6k+3}},\quad
    Q_4(z) = \prod_{k\ge0}\frac1{\lpa{1-z^{8k+3}}
    \lpa{1-z^{8k+4}}\lpa{1-z^{8k+5}}},\\
    Q_5(z) &= \prod_{k\ge0}\frac1{\lpa{1-z^{10k+3}}
    \lpa{1-z^{10k+4}}\lpa{1-z^{10k+5}}^2
    \lpa{1-z^{10k+6}}\lpa{1-z^{10k+7}}}.
\end{align*}
For a BPP $f$ with $m \ge n$, the function $g$ on $\mathbb{N}^2$
defined by $g(i,j) = f(i+j, j)$ is a plane partition, and by
replacing each row of $g$ (which is an integer partition) by its
conjugate partition, we obtain a column-strict plane partition
(weakly decreasing in each row but strictly decreasing in each
column). This transformation is clearly bijective. An example is
given in Figure~\ref{fig:strict-bij}.

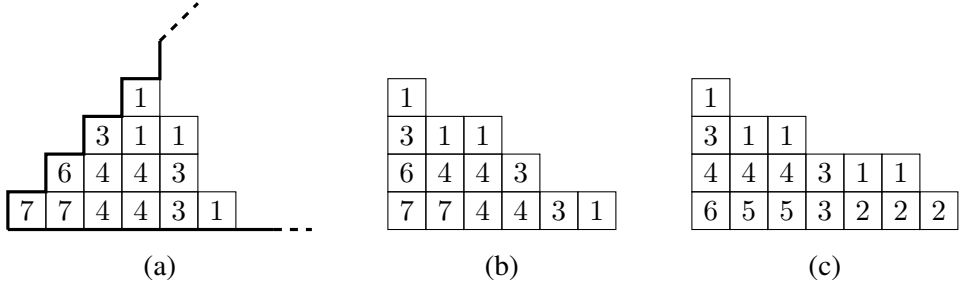
\begin{figure}[!thbp]
\centering
\begin{tikzpicture}
\begin{scope}
\draw (0,0) \mysqr{$7$};
\draw (0.5,0) \mysqr{$7$};
\draw (1,0) \mysqr{$4$};
\draw (1.5,0) \mysqr{$4$};
\draw (2, 0) \mysqr{$3$};
\draw (2.5, 0) \mysqr{$1$};
\draw (0.5,0.5) \mysqr{$6$};
\draw (1,0.5) \mysqr{$4$};
\draw (1.5,0.5) \mysqr{$4$};
\draw (2,0.5) \mysqr{$3$};
\draw (1,1) \mysqr{$3$};
\draw (1.5,1) \mysqr{$1$};
\draw (2,1) \mysqr{$1$};
\draw (1.5,1.5) \mysqr{$1$};
\draw[line width=1.2] (0,0) -- (0,0.5) -- (0.5,0.5) -- (0.5,1) -- (1,1) -- (1,1.5) -- (1.5,1.5) -- (1.5,2) -- (2,2) -- (2,2.5);
\draw[dashed, line width=1.2] (2,2.5) -- (2.5,3);
\draw[line width=1.2] (0,0) -- (3.5,0);
\draw[dashed, line width=1.2] (3.5,0) -- (4,0);
\node (a) at (2,-0.5) {(a)};
\end{scope}
\begin{scope}[xshift=5cm]
\draw (0,0) \mysqr{$7$};
\draw (0.5,0) \mysqr{$7$};
\draw (1,0) \mysqr{$4$};
\draw (1.5,0) \mysqr{$4$};
\draw (2, 0) \mysqr{$3$};
\draw (2.5, 0) \mysqr{$1$};
\draw (0,0.5) \mysqr{$6$};
\draw (0.5,0.5) \mysqr{$4$};
\draw (1,0.5) \mysqr{$4$};
\draw (1.5,0.5) \mysqr{$3$};
\draw (0,1) \mysqr{$3$};
\draw (0.5,1) \mysqr{$1$};
\draw (1,1) \mysqr{$1$};
\draw (0,1.5) \mysqr{$1$};
\node (b) at (1.5,-0.5) {(b)};
\end{scope}
\begin{scope}[xshift=9cm]
\draw (0,0) \mysqr{$6$};
\draw (0.5,0) \mysqr{$5$};
\draw (1,0) \mysqr{$5$};
\draw (1.5,0) \mysqr{$3$};
\draw (2, 0) \mysqr{$2$};
\draw (2.5, 0) \mysqr{$2$};
\draw (3,0) \mysqr{$2$};
\draw (0,0.5) \mysqr{$4$};
\draw (0.5,0.5) \mysqr{$4$};
\draw (1,0.5) \mysqr{$4$};
\draw (1.5,0.5) \mysqr{$3$};
\draw (2,0.5) \mysqr{$1$};
\draw (2.5,0.5) \mysqr{$1$};
\draw (0,1) \mysqr{$3$};
\draw (0.5,1) \mysqr{$1$};
\draw (1,1) \mysqr{$1$};
\draw (0,1.5) \mysqr{$1$};
\node (c) at (1.75,-0.5) {(c)};
\end{scope}
\end{tikzpicture}
\medskip
\caption{Example of the bijection between BPPs with $m \ge n$ and
column-strict plane partitions: (a) a BBP $f$ with $m \ge n$, (b)
the associated plane partition $g$, (c) the column-strict plane
partition obtained by taking the conjugate partition of each row of
$g$.} \label{fig:strict-bij}
\end{figure}

The generating function of column-strict plane partitions is known to 
be of the form 
\[
	\prod_{k\ge1}\frac1{(1-z^k)^{\tr{(k+1)/2}}};
\]
see \cite{Gordon1968,Stanley1971} or 
\cite[\href{https://oeis.org/A003293}{A003293}]{oeis2019}.

Based on the generating function \eqref{eq:han-xiong}, Han and Xiong
showed in \cite{Han2019}, by an elementary convolution approach
developed in \cite{Han2018}, that the number $G_{n,m}$ of BPPs of 
size $n$ and width $m$ satisfies
\begin{align}\label{eq:hx1}
    G_{n,m} \sim c(m) n^{-1} e^{\beta(m) \sqrt{n}},
\end{align}
for large $n$ and bounded $m\ge1$, where 
\[
    (c(m),\beta(m)) := \llpa{
	\frac{\sqrt{m^2+m+2}}{2^{(m^2-3m+14)/4} \sqrt{3m}}
	\prod\limits_{3\le j<m}
	\sin \Lpa{\frac{j\pi}{2m}}^{-\tr{(j-1)/2}},
	\sqrt{\frac{m^2+m+2}{6m}}\,\pi}.
\]
Thus $\log G_{n,m}$ is still of asymptotic order $\sqrt{n}$ when $m$ 
is bounded. Note that $c(1)=c(2)=1/(4\sqrt{3})$ and 
$\beta(1)=\beta(2)=\sqrt{2}\,\pi/\sqrt{3}$, the same as $c$ and 
$\beta$ in \eqref{eq:pn-as}, respectively. 

Now if we pretend that the formula \eqref{eq:hx1} holds also for
increasing $m$, then since $\beta(m)\sim \sqrt{m/6}\,\pi$ for large
$m$, we see that $\beta(m)\sqrt{n}\asymp \sqrt{mn}\asymp n^{2/3}$
when $m\asymp n^{1/3}$ (where the Hardy symbol $a_n\asymp b_n$ stands
for \emph{equivalence of growth order} for large $n$, equivalent to
the Bachmann-Laudau notation $a_n = \Theta(b_n)$; see
\cite{Knuth1976}). Furthermore, we will show in
Proposition~\ref{prop:Cm} that $\log c(m) \sim
-\frac{7\zeta(3)}{8\pi^2}\,m^2$ for large $m$. Then equating
$m^2\asymp\sqrt{mn}$ also gives $m\asymp n^{1/3}$. Thus we would
expect that \eqref{eq:hx1} remains valid for $m=o\lpa{n^{1/3}}$ and
the ``phase transition" occurs around $m\asymp n^{1/3}$. However,
while the latter is true by such a heuristic reasoning, the former is
not as we will prove that \eqref{eq:hx1} holds indeed only when
$m=o\lpa{n^{1/7}}$, although the weaker asymptotic estimate $\log
G_{n,m}\sim \beta(m)\sqrt{n}$ does hold uniformly for $1\le
m=o\lpa{n^{1/3}}$ (see \eqref{eq:Gnm-small} and \eqref{eq:n-1over7}).
This implies particularly the estimate
\begin{equation}\label{eq:logG-small}
    \log G_{n,m}\sim \frac{\pi}{\sqrt{6}}\,\sqrt{mn},
\end{equation}
which holds uniformly when $m\to\infty$, $m=o\lpa{n^{1/3}}$.

On the other hand, Gordon and Houten \cite{Gordon1969} showed that 
\begin{align}\label{eq:m-infty}
    G_{n,n} = [z^n]\prod_{k\ge1}\frac1{(1-z^k)^{\tr{(k+1)/2}}}
    \sim cn^{-49/72} e^{\beta_1 n^{2/3}+\beta_2 n^{1/3}},
\end{align}
where
\begin{align}\label{eq:c-b-b}
    (c,\beta_1,\beta_2)
    = \Lpa{\frac{e^{\zeta'(-1)/2 - \pi^4/(3456\zeta(3))} 
    \zeta(3)^{13/72}}{2^{3/4} (3\pi)^{1/2}},
    \frac{3\zeta(3)^{1/3}}{2}, 
    \frac{\pi^2}{24\zeta(3)^{1/3}}}.
\end{align}
This implies particularly the weak asymptotic estimate
\begin{align}\label{eq:logG-m-large}
	\log G_{n,n} \sim \frac{3\zeta(3)^{1/3}}{2}\, n^{2/3}.
\end{align}

In Section~\ref{sec:pt} we will derive stronger asymptotic
approximations to $G_{n,m}$ for all possible values of $m$, $1\le
m\le n$, covering \eqref{eq:hx1} and \eqref{eq:m-infty} as special
cases. In particular, as far as log-asymptotics is concerned, we
derive a uniform estimate, covering also the most interesting
critical range when $m\asymp n^{1/3}$; see
Proposition~\ref{prop:critical}. Define
\begin{equation}\label{eq:eta-def}
    \eta_d^{}(z) \defeq 
    \sum_{\ell \ge 1} \frac{e^{-\ell z}}
    {\ell^{2d-1}(1+e^{-\ell z})}
    \qquad(d\in\mathbb{N}; \Re(z)>0).
\end{equation}
\begin{wrapfigure}[14]{r}{0.25\textwidth}
    \begin{center}
    \includegraphics[width=0.9\linewidth]{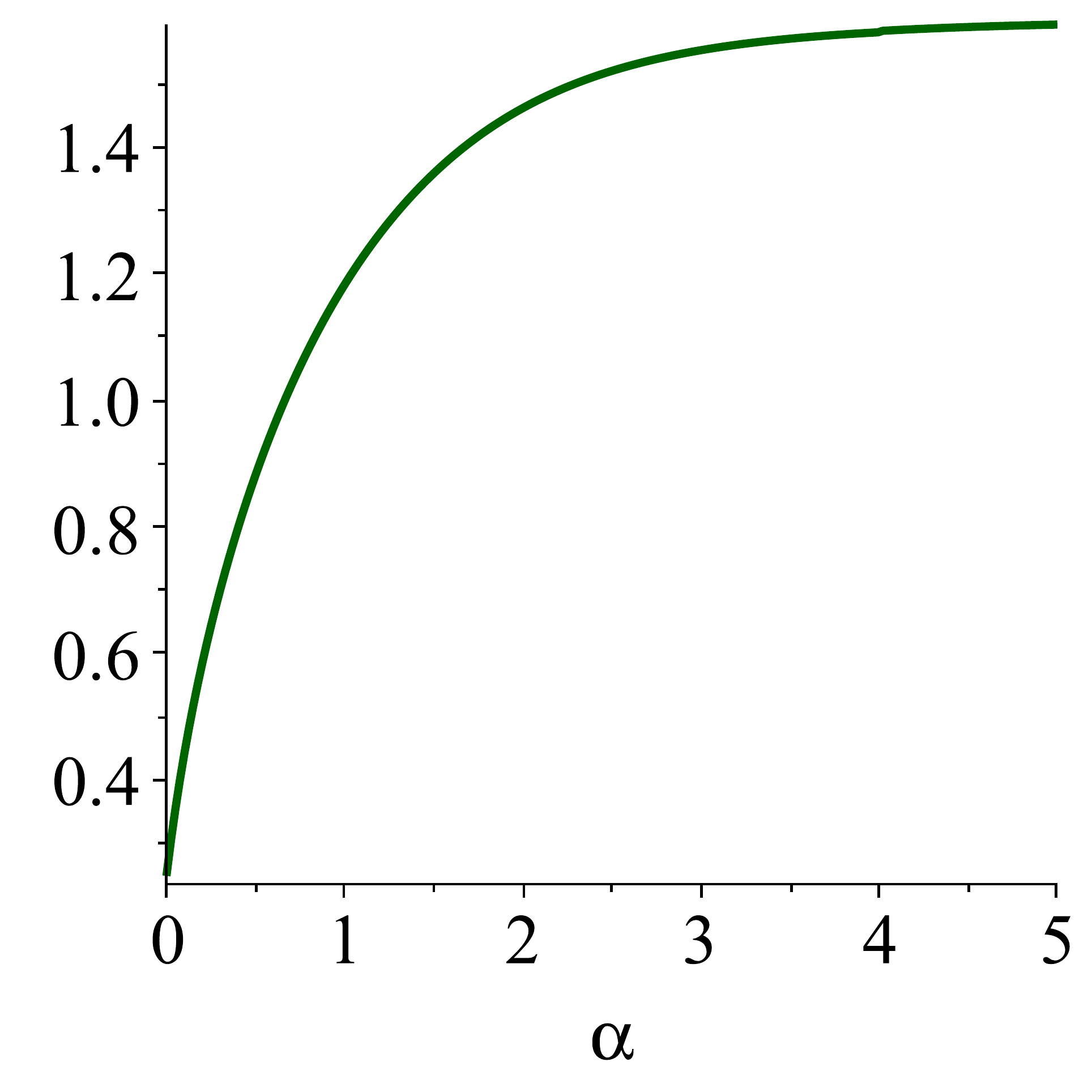}
    \end{center}
    \caption{A plot of the increasing function $G(\alpha)$.}
\end{wrapfigure}

\begin{thm} \label{thm:log-pt}
Let $\alpha := mn^{-1/3}$. Then 
\begin{align}\label{eq:logG-ua}
    \frac{\log G_{n,m}}{n^{2/3}} 
    \sim G(\alpha) 
    := r + \frac{\zeta(3)- 2\eta_2(\alpha r)}{2r^2},
\end{align}
uniformly when $\alpha\gg n^{-1/3}$ (or $m\to\infty$), where     
$r=r(\alpha)>0$ solves the equation 
\begin{align}\label{eq:sdpt}
    r^3 -\zeta(3) 
    +2\eta_2(\alpha r)-\alpha r\eta_2'(\alpha r) = 0.
\end{align}
In particular,
\begin{align}\label{eq:C-alpha-as}
    G(\alpha)
    \sim \begin{dcases}
        \frac{\pi}{\sqrt{6}}\,\sqrt{\alpha}, 
        & \text{if }\alpha\to0;\\
        \frac32\,\zeta(3)^{1/3},
        & \text{if }\alpha\to\infty.\\
    \end{dcases}
\end{align}
\end{thm}

We thus have a combinatorial model that interpolates nicely between
integer partitions and column-strict plane partitions, in the sense of
asymptotic behavior. A very similar looking expression will be 
derived in Section~\ref{sec:m-rowed} for $m$-rowed plane partitions, 
which bridge particularly ordinary partitions and plane partitions. 

The BPPs we study here can be connected to ordinary plane partitions
through the following decomposition. Given a plane partition $g$ of
size $n$, denote by $t = \sum_{i \ge 0} g(i,i)$ its trace. We
separate $g$ by the diagonal $i=j$ for $(i,j) \in \mathbb{N}^2$,
obtaining two BPPs $f_1, f_2$ of sizes $n_1, n_2$ respectively, and
an integer partition on the diagonal, such that $n = n_1 + n_2 + t$.
The weak asymptotics of such a triple $(n_1, n_2, t)$ is bounded
above by
\begin{align*}
    \log G_{n_1, n_1} & + \log G_{n_2,n_2} + \log p_t  \\ 
    &\le \beta_1(n_1^{2/3} + n_2^{2/3}) 
    + \beta_2(n_1^{1/3} + n_2^{1/3}) + O(\sqrt{t}+\log n) \\
    &\le 2^{1/3} \beta_1 n^{2/3} 
    -  2^{-2/3} \beta_1 n^{-1/3} t 
    + 2^{2/3} \beta_2 n^{1/3} + O(\sqrt{t}+\log n),
\end{align*}
with $\beta_1, \beta_2$ defined in \eqref{eq:c-b-b}. The last
inequality uses the concavity of $x \mapsto x^{2/3}$ and the fact
that $(1-x)^{2/3} \le 1 - x/2$ for $0 \le x \le 1$. Since $t = O(n)$,
the dominant term of the last upper bound matches that in
\eqref{eq:pp}. If $t = \omega(n^{2/3})$, the subdominant term will be
negative and of order $\Theta(n^{-1/3}t)$, making the bound
exponentially smaller than \eqref{eq:pp}. The main contribution thus
comes from $t = O(n^{2/3})$. This is consistent with the results in
\cite{Kamenov2007} on the asymptotic normality of $t$, with mean
asymptotic to $c_1 n^{2/3}$ and variance to $c_2n^{2/3}\log n$ for
some explicit constants $c_1$ and $c_2$.

For the method of proofs, we will employ a more classical approach
based on Mellin transforms (see \cite{Flajolet1995}) and saddle-point
method (see \cite{Andrews1976, Flajolet2009, Meinardus1954}), instead
of the elementary approach used in \cite{Han2018,Han2019}, which
becomes cumbersome when finer asymptotic expansions are required. The
analytic approach we adopted, although standard as that presented in
\cite{Andrews1976, Meinardus1954}, which applies for fixed $m$,
becomes more delicate because we address the whole range $1\le m\le
n$, and describing the transitional behaviors in different ``phases"
requires a finer analysis by maintaining particularly the uniformity
of all error terms involved with varying $m$.

Of additional interest here is that, similar to the functional
equation satisfied by the generating function of $p_n$
\begin{align}\label{eq:p-modular}
    P(e^{-\tau})
	:=\sum_{n\ge0}p_ne^{-n\tau}
    = \sqrt{\frac{\tau}{2\pi}}\,
    \exp\Lpa{\frac{\pi^2}{6\tau}-\frac{\tau}{24}}
    P\lpa{e^{-4\pi^2/\tau}}\qquad(\Re(\tau)>0),
\end{align}
(see \cite{Ayoub1963}), we also have the following (non-modular) 
relation satisfied by the generating function of $G_{n,m}$. 

\begin{thm} \label{thm:q-expr-exact}
For $\Re(\tau) > 0$, the function $G_m(e^{-\tau})$ satisfies the 
identity
\begin{align}\label{eq:q-expr-exact}
    G_m(e^{-\tau}) 
    = g_m \sqrt{\tau}\, 
    \exp\Lpa{\frac{\varpi_m}{\tau}+\phi_m\tau}
    K_m\lpa{e^{-4\pi^2/\tau}}L_m\lpa{e^{-4\pi^2/\tau}},
\end{align}
where the constants depending on $m$ are given by 
\begin{equation}\label{eq:f-details}
    \left\{
    \begin{split}
        g_m &:= (2\pi)^{-(m^2-3m+4)/4}
        \prod_{1 \le k< j< m}
        \Gamma\Lpa{\frac{k+j}{2m}},  \\
        \varpi_m &:= \frac{\pi^2}{24}
		\Lpa{m+1+\frac 2m}, \quad
        \phi_m := \frac{m^3-7m+2}{96},
    \end{split}\right.
\end{equation}
and the two functions $K_m$ and $L_m$ by
\begin{equation}\label{eq:XY}
    \left\{
    \begin{split}       
        K_m(z) &:= \sqrt{\frac{P\lpa{z^{1/m}}}
		{P\lpa{z^{1/2}}}}\,P\lpa{z}^{(m+2)/4},\\
        L_m(z) &:= \exp\llpa{ - \frac1{2m}
        \sum_{1 \le \ell <m}
        \frac{\cos\lpa{\frac{(2\ell-1)\pi}{m}}}
        {1-\cos\lpa{\frac{(2\ell-1)\pi}{m}}} \sum_{j \ge 0}
        \frac{z^{j+\frac{2\ell-1}{2m}}}
        {\lpa{j + \frac{2\ell-1}{2m}}
        \lpa{1-z^{j+\frac{2\ell-1}{2m}}}}}.
    \end{split}\right.
\end{equation}
Both $K_m(z)$ and $L_m(z)$ are analytic in $|z|<1$, $z\not\in[-1,0]$. 
\end{thm}
The expression \eqref{eq:q-expr-exact} is complicated but exact, and 
is the basis of our saddle-point analysis for characterizing the  
asymptotic behaviors of $G_{n,m}$. It is derived by Mellin transforms 
and the functional equation for the Hurwitz zeta function; see 
\cite[\S 12.9]{Apostol1976}. Note that 
\[
    Q_3(z) = \prod_{k\ge0}\frac1{1-z^{6k+3}}
    = \frac{P(z^3)}{P(z^6)}
    =\prod_{k\ge1}\lpa{1+z^{3k}},
\]
so we also have, by \eqref{eq:p-modular}, the functional equation
\[
    Q_3(e^{-\tau})
    = \frac{e^{\pi^2/(36\tau)+\tau/8}}
    {\sqrt{2}\,Q_3\lpa{e^{-2\pi^2/(9\tau)}}}.
\]
No such equation is available for higher $Q_m(z)$ with $m\ge4$. On 
the other hand, the sequence $G_{n,3}$ coincides with 
\href{https://oeis.org/A266648}{A266648} in OEIS \cite{oeis2019}. 

The rest of this paper is structured as follows. The exact expression
of $G_m$ in Theorem~\ref{thm:q-expr-exact} is first proved in the
next section. Then we turn to the asymptotics of $G_m$ in
Section~\ref{sec:asym-constant}. A uniform asymptotic approximation
to $G_{n,m}$ is then derived in Section~\ref{sec:Gnm}, which is used
in Section~\ref{sec:pt} to characterize the more precise behaviors of
$G_{n,m}$ in each of the three phases: sub-critical, critical and
super-critical. We then extend the same approach in
Section~\ref{sec:m-rowed} to $m$-rowed plane partitions, together
with two other similar variants.

\textbf{Notations.} Since $Q_m(z)=1$ for $m\le2$, we assume 
$m\ge3$ throughout this paper. The symbols $c,c',\beta$ and $c_j, 
\beta_j$ are generic whose values will always be locally specified. 
Other symbols are global except otherwise defined (e.g., in 
Section~\ref{sec:m-rowed}). 

\section{Exact expression for $G_{m}(e^{-\tau})$: proof of Theorem~\ref{thm:q-expr-exact}} \label{sec:exact-expr}

In this section, we will prove Theorem~\ref{thm:q-expr-exact} for the
exact expression \eqref{eq:q-expr-exact} for $G_{m}(e^{-\tau})$ by
Mellin transforms. We start with rewriting $Q_m(z)$ in
\eqref{eq:han-xiong} as
\begin{align}\label{eq:q}
    Q_m(z) = \prod_{k \ge 0} \,
	\prod_{1 \le j < 2m} 
	\left(\frac1{1-z^{2mk+j}}\right)^{w_m(j)}, 
\end{align}
where
\begin{align}\label{eq:wmj}
    w_m(j) 
	\defeq \Bigl\lfloor 
	\frac{m - 1 - |m - j|}{2} \Bigr\rfloor
	\qquad(1\le j<2m).
\end{align}
For convenience, the $k$th moment of $w_m$ is denoted by $\mu_k(w_m)$:
\[
	\mu_k = \mu_k(w_m) 
	\defeq \sum_{1 \le j < 2m} j^k w_m(j)
    \qquad(k\in\mathbb{N}).
\]
By considering the parity of $j$ and $m$, we deduce that 
\begin{align}\label{eq:Wmz}
	W_m(z) 
	:= \sum_{1\le j<2m}w_m(j) z^j
	= \frac{z^3(1-z^{m-1})(1-z^{m-2})}{(1+z)(1-z)^2}
	\qquad(m\ge3).
\end{align}
From this expression, it is straightforward to compute the first few 
moments $\mu_k = k![s^k]W_m(e^s)$, as given explicitly in 
Table~\ref{tab:mu}. 
\begin{table}[!ht]\def\arraystretch{2}
    \begin{center}\small
    \begin{tabular}{cccc}
    $\mu_0$ & $\mu_1$ & $\mu_2$ & $\mu_3$\\[-4pt] \hline
    $\displaystyle\frac{(m-1)(m-2)}{2}$ & 
    $\displaystyle\frac{m(m-1)(m-2)}{2}$ &
    $\displaystyle\frac{m(m-1)(m-2)(7m-3)}{12}$ &
    $\displaystyle\frac{3m^2(m-1)^2(m-2)}{4}$\\
    \end{tabular}    
    \end{center}
    \caption{The exact expressions of $\mu_k$ for $0\le k\le 3$.}
    \label{tab:mu}
\end{table}

Since all singularities of $G_m(z)$ lie on the unit circle, we
consider the change of variables $z = e^{-\tau}$ and examine the
behavior of $G_m(e^{-\tau})$ in the half-plane $\Re(\tau)>0$. For 
that purpose, let 
\[
    \zeta(s,b) := \sum_{k\ge0}(k+b)^{-s}
	\qquad(\Re(s)>1,b>0)
\]
denote the Hurwitz zeta function. In addition to Mellin transforms,
we need some properties of $\zeta(s, b)$ and the Gamma function 
$\Gamma(s)$; see, for example, \cite[Ch.~12]{Apostol1976},
\cite[Ch.~1]{Erdelyi1953} or \cite[Chs. XII \& XIII]{Whittaker1996}.
Since $P(e^{-\tau})$ satisfies \eqref{eq:p-modular}, we need only
derive a similar expression for $Q_m(e^{-\tau})$ in order to prove
\eqref{eq:q-expr-exact}.

\begin{prop} \label{prop:qm-expr}
For $\Re(\tau)>0$, $q_m(e^{-\tau})\defeq \log Q_m(e^{-\tau})$ 
satisfies 
\begin{align} 
    \begin{split} \label{eq:f-expansion}
        q_m(e^{-\tau}) 
        &= \frac{(m-1)(m-2)\pi^2}{24m\tau} 
        + \sum_{1 \le j < 2m} w_m(j)
        \log\Gamma\Lpa{\frac{j}{2m}}\\
        &\qquad- \frac{(m-1)(m-2)}{4} \log(2\pi) 
        + \frac{(m-1)(m-2)(m+3)}{96}\,\tau
        + E(\tau),
    \end{split}
\end{align}
where $E(\tau)$ is given by 
\begin{equation}\label{eq:E0-def}
    E(\tau) = E(m;\tau) 
    \defeq \frac1{2\pi i}\int_{(-2)} 
    \Gamma(s) \zeta(s+1) 
    \mathscr{M}_m(s) \tau^{-s}\dd s,
\end{equation}
with $\int_{(c)}$ representing $\int_{c-i\infty}^{c+i\infty}$ and 
\begin{equation} \label{eq:mstar-def}
    \mathscr{M}_m(s) 
    \defeq (2m)^{-s} \sum_{1 \le j \le 2m} 
    w_m(j) \zeta\Lpa{s, \frac{j}{2m}}.
\end{equation} 
\end{prop}
\begin{proof}
Let $\mathscr{M}_m^{[q]}(s)$ be the Mellin transform of
$q_m(e^{-\tau})$. Then $\mathscr{M}_m^{[q]}(s) = \Gamma(s) \zeta(s+1)
\mathscr{M}_m(s)$ for $\Re(s)>1$, where $\mathscr{M}_m(s)$ is defined
in \eqref{eq:mstar-def}. By the inverse Mellin transform, we have
\begin{align}\label{eq:mellin-main}
	q_m(e^{-\tau})
	=\frac1{2\pi i}\int_{(r)}\mathscr{M}_m^{[q]}(s) \tau^{-s} \dd s
	\qquad(r>1).
\end{align}
We will move the line of integration to the left, so as to include
the leftmost pole at $s=-1$, and collect all the residues of the
poles encountered. For that purpose, we need the growth properties of
the integrand at $c\pm i\infty$ to ensure the absolute convergence of
the integral.

By the known estimate for Gamma function (see 
\cite[\S 1.18]{Erdelyi1953})
\[
    |\Gamma(c+it)| = O\lpa{|t|^{c-1/2} 
    e^{-\pi|t|/2}},\qquad(c\in\mathbb{R}, |t|>1),
\]
and that for Hurwitz zeta function (see \cite[\S 
13.51, p.\ 276]{Whittaker1996})
\begin{equation} \label{eq:zeta-bound}
	\begin{split}
	    |\zeta(c+it, b)| = O\lpa{|t|^{\nu_0(c)}\log |t|},
		\;\text{ with }\;
		\nu_0(c):=
		\begin{dcases}
			\tfrac12-c, &\text{if }c<0;\\
			\tfrac12, &\text{if }c\in[0,\tfrac12];\\
			1-c, &\text{if } c\in[\tfrac12,1];\\
			0, &\text{if }c>1,
		\end{dcases}
	\end{split} 
\end{equation}
for $|t|>1$, we have
\begin{equation}
	|\mathscr{M}_m^{[q]}(c+it)\tau^{-s}| 
	= O\lpa{m^{2-c} |t|^{\nu(c)}(\log|t|)^2
	e^{-\frac{\pi}{2}|t| + t\arg(\tau)}},
\end{equation}
for $c\in\mathbb{R}, |t|>1$, where
\[
	\nu(c) := \begin{dcases}
		\tfrac12+|c|, 
		& \text{if }|c-\tfrac12|\ge\tfrac12;\\
		\min\{\tfrac12+c,\tfrac32-c\}, 
		&\text{if }|c-\tfrac12|\le \tfrac12.
	\end{dcases}
\]

\begin{center}
	\begin{tabular}{cc}
		\includegraphics[width=4.5cm]{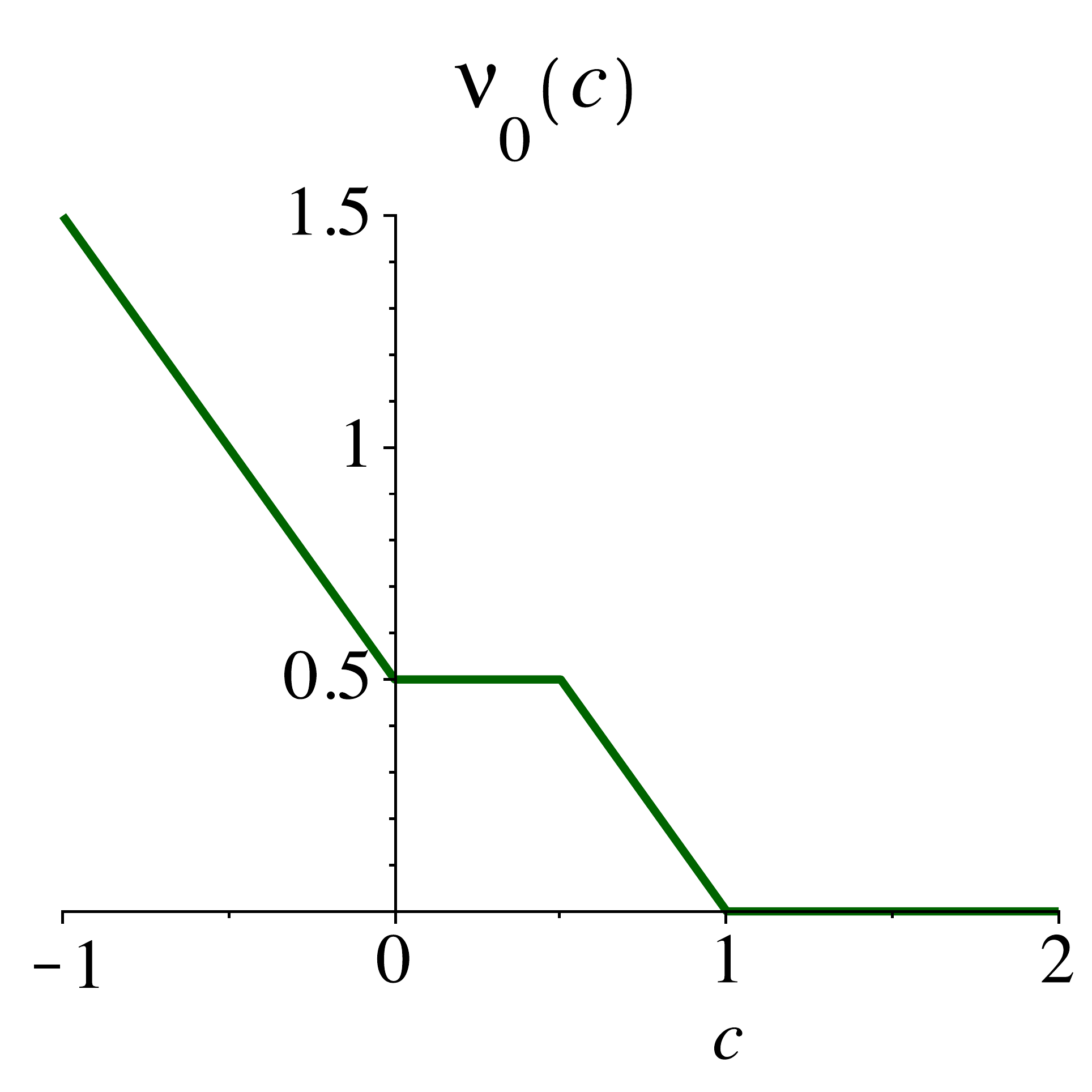} &
		\includegraphics[width=4.5cm]{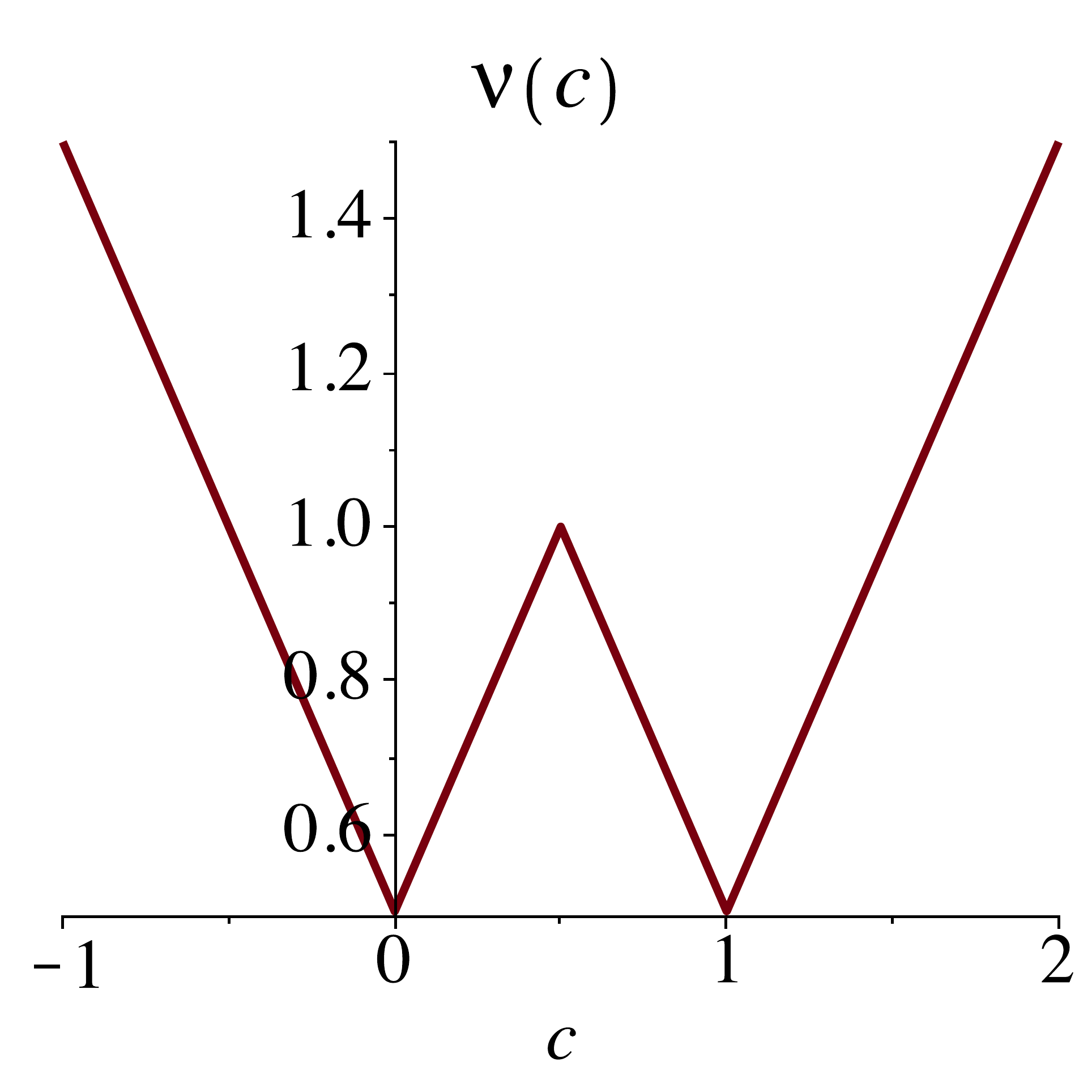}
		\\
		$\nu_0(c)$ & $\nu(c)$
	\end{tabular}
\end{center}
Thus the integral in \eqref{eq:mellin-main} is absolutely convergent
as long as $|\arg(\tau)| \le \pi/2 - \ve$, and this justifies the
analytic properties we need for summing the residues, which we now
compute. Since $w_m(j) = w_m(2m-j)$ (see \eqref{eq:wmj}), we can
rewrite \eqref{eq:mstar-def} as
\begin{equation} \label{eq:mstar}
	\mathscr{M}_m(s) 
	= \sum_{1 \le j <m} w_m(j) 
	\Lpa{\zeta\Lpa{s, \frac{j}{2m}} 
	+ \zeta\Lpa{s, 1-\frac{j}{2m}}}
	+ w_m(m) \zeta\Lpa{s,\frac1{2}}.
\end{equation}
Observe that $\mathscr{M}_m(-2j)=0$ for $j\in\mathbb{Z}^+$ because
$\zeta(-2j, x) = -B_{2j+1}(x)/(2j+1)$, where $B_{2j+1}(x)$ is the
Bernoulli polynomial of order $2j+1$:
\begin{align}\label{eq:bern-poly}
	B_j(x) := j![z^j]\frac{ze^{xz}}{e^z-1},
\end{align}
which satsfies $B_{2j+1}(x) = -B_{2j+1}(1-x)$; see
\cite[\S~1.13]{Erdelyi1953}. On the other hand, $\zeta(s+1)=0$ when
$s<-1$ is odd. Thus the only poles of the integrand in
\eqref{eq:mellin-main} are $s=1$ (simple), $s=0$ (double) and $s=-1$
(simple); this similarity to that of $\log P(e^{-\tau})$ suggests the
possibility of the identity \eqref{eq:q-expr-exact}.

From these properties, it follows that
\begin{align}\label{eq:q-expansion-res}
	\begin{split}
		q_m(e^{-\tau}) 
		= \sum_{-1\le j\le 1}\Res_{s=j}
        \lpa{\Gamma(s) \zeta(s+1) 
        \mathscr{M}_m(s)\tau^{-s}}+ E(\tau),
	\end{split}
\end{align}
where $E(\tau)$ is as defined in \eqref{eq:E0-def}. By the local 
expansions of $\Gamma(s)$, $\zeta(s+1)$ and $\zeta(s,b)$ for $s\sim0$ 
(see \cite{Erdelyi1953}):
\begin{align*}
	\Gamma(s) 
	&= \frac1{s} - \gamma + O(|s|), \quad \zeta(s+1) 
	= \frac1{s} + \gamma + O(|s|), \\
	\zeta(s, b) 
	&= \frac1{2} - b+ \Lpa{\log\Gamma(b) 
	- \frac1{2} \log(2\pi)} s + O(|s|^2),
\end{align*}
where $\gamma$ is the Euler-Mascheroni constant, we then have
\[
    \begin{split} 
        q_m(e^{-\tau}) 
        = \frac{\pi^2 \mu_0}{12m\tau} 
        &+ \sum_{1 \le j < 2m} w_m(j)
        \log\Gamma\Lpa{\frac{j}{2m}}
        - \frac{\mu_0}{2} \log(2\pi) 
        + \left( - \frac{\mu_2}{8m} 
        + \frac{\mu_1}{4} - \frac{m\mu_0}{12}\right) \tau 
        + E(\tau).
    \end{split}
\]
This, together with the expressions in Table~\ref{tab:mu}, proves 
\eqref{eq:mellin-main}.
\end{proof}

We now evaluate $E(\tau)$, beginning with a simple lemma.

\begin{lem} \label{lem:sum-sine}
For integers $m > 1$, $1 \le \ell \le 2m$ and real number $\theta$, 
we have
\[
	\sum_{1\le k < j <m}
	\sin\Lpa{\theta + \frac{\ell(k+j)\pi}{m}} 
	= \sin \theta \times
	\begin{dcases}
		\binom{m-1}{2}, & \text{for } \ell = 2m; \\
		-\Bigl\lfloor \frac{m-1}2 \Bigr\rfloor, 
		& \text{for } \ell = m; \\
		1, & \text{for } 1 \le \ell < 2m; 
		\ell \neq m \text{ and } \ell \text{ even;} \\
		-\frac{\cos(\ell\pi/m)}{1-\cos(\ell\pi/m)}, 
		& \text{for } 1 \le \ell < 2m, \ell \neq m \text{ and } 
		\ell \text{ odd.}
	\end{dcases}
\]
\end{lem}
\begin{proof}(Sketch) We consider the identity
\[
    \llpa{\sum_{1\le k <m} 
    \exp\Lpa{\frac{k\ell \pi i}{m}}}^2
    = 2\sum_{1 \le k < j <m} 
	\exp\Lpa{\frac{(k+j) \ell \pi i}{m}}
    \times \sum_{1 \le k <m} 
	\exp\Lpa{\frac{2k\ell \pi i}{m}},
\]
and perform straightforward simplifications in each case. 
\end{proof}

We now compute the error term $E(\tau)$. Let $p(z) := \log P(z)$. 
\begin{prop} \label{prop:E0-est}
The error term $E(\tau)$ defined in (\ref{eq:E0-def}) satisfies
\[
	E(\tau) = \kappa_m\lpa{e^{-4\pi^2/\tau}} 
	- p\lpa{e^{-4\pi^2/\tau}}
	+ \lambda_m\lpa{e^{-4\pi^2/\tau}} ,
\]
for $\Re(\tau)>0$, where ($K_m, L_m$ defined in \eqref{eq:XY})
\begin{align}
	\kappa_m(z) &:= \log K_m(z)
	=\frac{m+2}{4} p(z) 
	+ \frac1{2}p\lpa{z^{1/m}} 
	- \frac1{2}p\lpa{z^{1/2}}
	\label{eq:E02},\\ 
	\lambda_m(z) &:= \log L_m(z)
	= - \frac1{2m} \sum_{1 \le \ell <m} 
	\frac{\cos\lpa{\frac{(2\ell-1)\pi}{m}}}
	{1-\cos\lpa{\frac{(2\ell-1)\pi}{m}}} 
	\sum_{k \ge 0} 
	\frac{z^{k + \frac{2\ell-1}{2m}}}{\lpa{k + 
	\frac{2\ell-1}{2m}}
	\lpa{1-z^{k + \frac{2\ell-1}{2m}}}} 
	\label{eq:E01}.
\end{align}
\end{prop}
\begin{proof}
We first rewrite the single-sum relation \eqref{eq:mstar-def} for  
$\mathscr{M}_m(s)$ as a double sum:
\[
    \mathscr{M}_m(s) 
	= (2m)^{-s} \sum_{1 \le h < j <m} 
	\zeta\Lpa{s, \frac{h+j}{2m}}.
\]
Combining this with the functional equation for the Hurwitz zeta 
function (see \cite[\S 12.9]{Apostol1976})
\begin{equation}\label{eq:hurwitz-zeta-eq}
	\zeta\Lpa{s,\frac{j}{d}} 
	= \frac{2\Gamma(1-s)}{(2d\pi)^{1-s}} 
	\sum_{1 \le \ell \le d} 
	\sin\Lpa{\frac{\pi s}{2} + \frac{2\ell j \pi}{d}}
	\zeta\Lpa{1-s, \frac{\ell}{d}}\qquad(d=1,2,\dots),
\end{equation}
we then have
\[
	\mathscr{M}_m(s) 
	= \frac{\Gamma(1-s)}{m(2\pi)^{1-s}} 
	\sum_{0\le \ell \le 2m}
	\zeta\Lpa{1-s, \frac{\ell}{2m}}
	\sum_{1 \le k < j<m} \sin\Lpa{\frac{\pi s}{2} 
	+ \frac{\ell (k+j)\pi}{m}} .
\]
Now, by Lemma~\ref{lem:sum-sine}, the sum above can be reduced to
\begin{align*}
	\mathscr{M}_m(s) 
	&= \frac{\Gamma(1-s)}{m(2\pi)^{1-s}} 
	\sin\left(\frac{\pi s}{2}\right)
	\bigg[ \binom{m-1}{2} \zeta(1-s) 
	- \Bigl\lfloor \frac{m-1}{2} \Bigr\rfloor 
	\zeta\Lpa{1-s,\frac1{2}} \\
	&\quad + \sum_{1 \le \ell < m, 2\ell \neq m} 
	\zeta\Lpa{1-s, \frac{\ell}{m}}
	- \sum_{1\le\ell<m, 2\ell-1\neq m} 
	\frac{\cos\lpa{\frac{(2\ell-1)\pi}{m}}}
	{1-\cos\lpa{\frac{(2\ell-1)\pi}{m}}} \,
	\zeta\Lpa{1-s, \frac{2\ell-1}{2m}} \bigg].
\end{align*}
Then, by the relation
\begin{equation}\label{eq:hurwitz-sum}
    \sum_{1\le \ell \le d} 
	\zeta\Lpa{s, \frac{\ell}{d}}
	= d^s \zeta(s)\qquad(d=2,3,\dots),
\end{equation}
which implies, in particular, $\zeta(s,1/2) = (2^s-1)\zeta(s)$, we 
deduce that 
\begin{align*}
	\mathscr{M}_m(s) 
	&= \frac{\Gamma(1-s)}{(2\pi)^{1-s}} 
	\sin\left(\frac{\pi s}{2}\right) 
	\bigg[ c(m,s)\zeta(1-s) - \frac1m
	\sum_{1\le\ell<m} \frac{\cos\lpa{\frac{(2\ell-1)\pi}{m}}}
	{1-\cos\lpa{\frac{(2\ell-1)\pi}{m}}} \,
	\zeta\Lpa{1-s, \frac{2\ell-1}{2m}} \bigg],
\end{align*}
where $c(m,s) := (m-2)/2 + m^{-s} - 2^{-s}$.

By applying the change of variables $s\mapsto -s$ in the integral 
representation in \eqref{eq:E0-def} of $E(\tau)$, we obtain 
\begin{equation} \label{eq:E0-mirror}
	E(\tau) 
	= \frac1{2\pi i}\int_{(2)}
	\Gamma(-s)\zeta(1-s)\mathscr{M}_m(-s)\tau^{s}\dd s.   
\end{equation}
Note that the functional equation \eqref{eq:hurwitz-zeta-eq} with  
$d=j=1$ implies for the Riemann zeta function that
\begin{equation}\label{eq:zeta-eq}
    \zeta(s)
	= 2^s \pi^{s-1} \Gamma(1-s) \zeta(1-s) 
	\sin\Lpa{\frac{\pi s}{2}}.
\end{equation}
By this and Euler's reflection formula for the Gamma function
\begin{equation}\label{eq:gamma-reflection}
    \Gamma(s)\Gamma(1-s) 
	= \frac{\pi}{\sin(\pi s)},
\end{equation}
we then get
\[
	\Gamma(-s)\zeta(1-s) 
	= -\frac{(2\pi)^{1-s}}{s\sin(\pi s)}\, \zeta(s) 
	\cos\Lpa{\frac{\pi s}{2}}.
\]

Consequently, the integrand in \eqref{eq:E0-mirror} can be written as
\begin{align*}
	\Gamma(-s)\zeta(1-s)
	&\mathscr{M}_m(-s) \tau^s 
	= \frac1{2} \left(\frac{4\pi^2}{\tau}\right)^{-s} 
	\Gamma(s) \zeta(s) \\
	&\quad \times \bigg[ c(m,-s)\zeta(1+s) 
	- m^{-1}\sum_{1\le\ell<m} 
	\frac{\cos\lpa{\frac{(2\ell-1)\pi}{m}}}
	{1-\cos\lpa{\frac{(2\ell-1)\pi}{m}}} \,
	\zeta\Lpa{1+s, \frac{2\ell-1}{2m}} \bigg].
\end{align*}
The two expressions \eqref{eq:E02} (contributed by terms involving 
$c(m,s)$) and \eqref{eq:E01} (contributed by terms involving the 
partial sum with the cosine functions) then follow from inverting the 
Mellin transform using the relation
\begin{align}\label{eq:JS}
	J(b,\tau):=\frac1{2\pi i}\int_{(c)}
	\Gamma(s)\zeta(s)\zeta(1+s,b) \tau^{s} \dd s
	= \sum_{k \ge 0} \frac{e^{-(k+b)/\tau}}
	{(k+b) \lpa{1-e^{-(k+b)/\tau}}},
\end{align}
for $\Re(\tau)>0$ and $b>0$, where $c>1$. In particular, the 
right-hand side equals $p\lpa{e^{-1/\tau}}$ when $b=1$. This 
completes the proof.
\end{proof}

\begin{proof}[Proof of Theorem~\ref{thm:q-expr-exact}]
Theorem~\ref{thm:q-expr-exact} is a direct consequence of a combination of
Proposition~\ref{prop:qm-expr}, Proposition~\ref{prop:E0-est} and
\eqref{eq:p-modular}.
\end{proof}

\section{Asymptotics of $\log G_m(e^{-\tau})$}
\label{sec:asym-constant}

We derive the asymptotic behavior of $\log G_m(e^{-\tau})$ as
$m\to\infty$ and $|\tau|\to0$. From Theorem~\ref{thm:q-expr-exact} 
and Proposition~\ref{prop:E0-est}, we have
\begin{equation} \label{eq:rm-expr}
      \log G_m(e^{-\tau}) 
      = \frac{\varpi_m}{\tau} 
      + \frac1{2}\log \tau 
      + \log g_m 
      + \phi_m\tau 
      + \kappa_m\lpa{e^{-4\pi^2/\tau}}
      + \lambda_m\lpa{e^{-4\pi^2/\tau}}, 
\end{equation}
for $\Re(\tau)>0$. Since $\kappa_m(z)$ depends only on $p(z)$ (see 
\eqref{eq:E02}), which, by \eqref{eq:p-modular}, satisfies
\begin{align}\label{eq:p-eq}
    p(e^{-\tau})
    = \frac{\pi^2}{6\tau}
    -\frac{\tau}{24}
    +\frac12\log\tau
    -\frac12\log(2\pi)
    +p\lpa{e^{-4\pi^2/\tau}}
    \qquad(\Re(\tau)>0),
\end{align}
so we need only to examine more closely the asymptotics of $\log g_m$
and $\lambda_m$ when $m$ is large and $|\tau|\to0$. Complications 
arise when $\tau$ may depend also on $m$.

\subsection{Asymptotics of $\log g_m$}

We now derive an asymptotic expansion for $\log g_m$ by the 
Euler-Maclaurin formula (see \cite[Ch.\ VIII]{Hardy1949}).
\begin{prop} \label{prop:Cm}
When $m\to\infty$, $\log g_m$ satisfies the asymptotic expansion
\begin{align}\label{eq:log-gm}  
	\log g_m 
	\sim - \frac{7\zeta(3)}{8\pi^2}\, m^2 
	+ \frac{11}{24}\, \log m 
    + c_1 - \sum_{j\ge1} \frac{B_{2j}B_{2j+2}(-\pi^2)^j}  
    {8j(j+1)(2j)!}\, m^{-2j},
\end{align}
where $c_1:=\frac{1}{2}\zeta'(-1) - \frac{11}{24}\log \pi - 
\frac{7}{24} \log 2$ and $B_j=B_j(0)$ denote the Bernoulli numbers.
\end{prop}
\begin{proof}
Starting from the definition of $g_m$ in \eqref{eq:f-details}, we write $\log g_m$ as
\[
	\log g_m 
	= - \frac{m^2-3m+4}4\, \log(2\pi)
    + S_m,
\]
where 
\[	
	S_m := \sum_{1 \le j <2m} w_m(j) 
	\log\Gamma\Lpa{\frac{j}{2m}}.
\]
Since $w_m(j) = w_m(2m-j)$, we have, by Euler's reflection formula 
\eqref{eq:gamma-reflection},
\begin{align*}
	S_m 
	&= \frac{\mu_0}{2}\, \log\pi 
	- \sum_{1 \le j <m} \Bigl\lfloor \frac{j-1}{2} 
      \Bigr\rfloor \log\Lpa{\sin\big(\frac{j\pi}{2m}\big)} \\
	&= \frac{(m-1)(m-2)}{4}\, \log\pi 
	- \sum_{1 \le j <m} \frac{j-1}{2} \log\Lpa{\sin\big(\frac{j\pi}{2m}\big)} 
    + \frac1{2} \sum_{1 \le j \le \lfloor m/2 \rfloor} 
	  \log\Lpa{\sin\big(\frac{j\pi}{m}\big)} \\
	&= \frac{(m-1)(m-2)}{4}\, \log\pi
	- \frac{S_{m,1}}2+\frac{S_{m,2}}2
    + \frac{S_{m,3}}2,
\end{align*}
where
\begin{align*}
	S_{m,1} 
    := \sum_{1 \le j \le m} 
	j \log\Lpa{\sin\big(\frac{j\pi}{2m}\big)},
	\quad
	S_{m,2} 
    := \sum_{1 \le j \le m} 
	\log\Lpa{\sin\big(\frac{j\pi}{2m}\big)},
	\quad
	S_{m,3} 
    := \sum_{1 \le j \le \lfloor m/2 \rfloor} 
	\log\Lpa{\sin\big(\frac{j\pi}{m}\big)}.
\end{align*}	
The last two sums are easily simplified by the elementary identity
\[
	\prod_{1\le j<k} \sin\Lpa{\frac{\pi j}{k}}
	= \frac{k}{2^{k-1}}\qquad(k=1,2,\dots),
\]
giving
\begin{equation}\label{eq:s2-res}
		S_{m,2} 
        = - (m-1)\log 2
        +\frac{\log m}{2} \quad\text{and}\quad
		S_{m,3} 
        = - \frac{m-1}{2} \log 2
        +\frac{\log m}{2}.
\end{equation}

We now evaluate $S_{m,1}$. By the local expansion $\log(\sin x) 
= \log x + O(x^2)$ for $x \to 0$, we decompose first the sum into two parts:
\[
	S_{m,1} 
	= \sum_{1 \le j \le m} j 
    \Lpa{\log\Lpa{\sin\big(\frac{j\pi}{2m}\big)} 
	- \log\Lpa{\frac{j\pi}{2m}}}
    + \sum_{1 \le j \le m} j \log\Lpa{\frac{j\pi}{2m}},
\]
and then we apply Euler-Maclaurin formula (see \cite[Ch.\ 
VIII]{Hardy1949}) to each sum, yielding  
\[
	\sum_{1 \le j \le m} j 
    \Lpa{\log\Lpa{\sin\big(\frac{j\pi}{2m}\big)} 
	- \log\Lpa{\frac{j\pi}{2m}}} 
	= c_2 m^2 
    - \frac m{2} \log\frac{\pi}{2} 
	- \frac1{12}\left(1 + \log\frac{\pi}{2}\right) 
	+ O(m^{-2}),
\]
where
\begin{align*}
	c_2 
	:= \frac1{m^2} \int_0^m x 
	\Lpa{\log\Lpa{\sin\big(\frac{x\pi}{2m}\big)}
    - \log\Lpa{\frac{x\pi}{2m}}} \dd x 
	= \frac{7\zeta(3)}{4\pi^2} 
    - \frac{\log\pi}{2} 
    + \frac1{4},
\end{align*}
and
\[
	\sum_{1 \le j \le m} j \log \Lpa{\frac{j \pi}{2m}}
	= \Lpa{\frac1{2} \log \frac{\pi}{2} 
	- \frac{1}{4}} m^2 
	+ \frac{m}{2} \log \frac{\pi}{2} 
	+ \frac{\log m }{12}+ \frac1{12} - \zeta'(-1) 
	+ O(m^{-2}).
\]
Summing up these two parts, we have
\begin{equation} \label{eq:s1-res}
	S_{m,1} = \Lpa{\frac{7\zeta(3)}{4\pi^2} 
	- \frac{\log 2}{2}}m^2 
	+ \frac{\log m}{12}
	- \left( \zeta'(-1) 
    + \frac1{12}\log\frac{\pi}{2} \right) 
	+ O(m^{-2}).
\end{equation}
By substituting \eqref{eq:s2-res} and \eqref{eq:s1-res} into
\begin{equation} \label{eq:s-sum-def}
    \log g_m 
	= - \frac1{2} \log\pi 
    - \frac{m^2-3m+4}{4}\, \log 2 
	- \frac{S_{m,1}}2+\frac{S_{m,2}}2
    + \frac{S_{m,3}}2,
\end{equation}
we obtain the expansion \eqref{eq:log-gm} up to an error of order 
$m^{-2}$. Further terms in \eqref{eq:log-gm} are computed by refining 
the expansion for $S_{m,1}$ following the same procedure and using 
the relation 
\[
    \frac{\text{d}^k}{\text{d}x^k}
    \,\log(\sin(x))\Bigl|_{x=\pi/2}
    =-\frac{\text{d}^{k-1}}{\text{d}x^{k-1}}
    \,\tan(x)\Bigl|_{x=\pi/2}
    = \frac{(2i)^k}{k}(2^k-1)B_k
    \qquad(k\ge2);
\]
see the OEIS sequence \cite[A155585]{oeis2019}. 
\end{proof}

\subsection{Asymptotics of $E(\tau)$}

We now consider the asymptotic behavior of the key ``calibrating"
term $E(\tau)$ defined in \eqref{eq:E0-def} as $\tau\to0$. This term
is asymptotically negligible when $m=o(n^{1/3})$, but plays a role
for larger $m$, notably in the transitional zone when $m\asymp
n^{1/3}$. We then need finer asymptotic approximations for $E(\tau)$,
which, by Proposition~\ref{prop:E0-est}, equals $E(\tau) =
\kappa_m\lpa{e^{-4\pi^2/\tau}}-p\lpa{e^{-4\pi^2/\tau}} + 
\lambda_m\lpa{e^{-4\pi^2/\tau}}$. We begin with the asymptotics of 
the first term, which is simpler.

\begin{coro} Assume $\Re(\tau)\to0$ in the half-plane $\Re(\tau)>0$. 
Then the function $\kappa_m$ satisfies
\begin{equation}\label{eq:kappa-asymp}
\begin{split}
	\kappa_m\lpa{e^{-4\pi^2/\tau}} 
    &= \frac12\,p\lpa{e^{-4\pi^2/(m\tau)}}
    +O\lpa{e^{-\Re(2\pi^2/\tau)}}\\
	&=\begin{dcases}
		O\lpa{e^{-\Re(4\pi^2/(m\tau))}}, 
		    &\text{if }m|\tau|\le 1,\\
		\frac{m\tau}{48}+\frac14\log\frac{2\pi}{m\tau}
        -\frac{\pi^2}{12m\tau}+\frac12p(e^{-m\tau})
            +O\lpa{e^{-\Re(2\pi^2/\tau)}},
		    &\text{if }m|\tau| \ge 1.
	\end{dcases}
\end{split}    
\end{equation}
\end{coro}
\begin{proof}
By \eqref{eq:E02}, we obtain the first relation in 
\eqref{eq:kappa-asymp}. On the other hand, the series
\begin{align*}
    p\lpa{e^{-4\pi^2/\tau}}
    = \sum_{j\ge1}\frac{e^{-4j\pi^2/\tau}}
    {j\lpa{1-e^{-4j\pi^2/\tau}}}
\end{align*}
is itself an asymptotic expansion when $|\tau|\to0$. The other 
estimate in \eqref{eq:kappa-asymp} when $m|\tau|\ge 1$ follows 
from the functional equation \eqref{eq:p-modular}.
\end{proof}

We now examine the other term $\lambda_m\lpa{e^{-4\pi^2/\tau}}$, 
beginning with the asymptotics of the integral $J(b,w)$ defined in 
\eqref{eq:JS}.
\begin{lem} If $b>0$, then 
\begin{align}\label{eq:JS-a}
    J(b,\tau) = 
    \begin{dcases}
        b^{-1}e^{-b/\tau}\lpa{1+O\lpa{e^{-\Re(b/\tau)}
        +e^{-\Re(1/\tau)}}}, & \text{as }|\tau|\to0;\\
        \zeta(2,b)\tau -\frac12\log\tau
        +\frac12\psi(b)+O(1),
        & \text{as }|\tau|\to\infty,
    \end{dcases}
\end{align}
uniformly in the half-plane $\Re(\tau)>0$, where $\psi$ is the 
digamma function defined by $\psi(x) = \Gamma(x)/\Gamma'(x)$. These estimates hold also when $b/|\tau|\to0$ and 
$b/|\tau|\to\infty$, respectively. 
\end{lem}
\begin{proof}
In the small $|\tau|$ case, the estimate follows from the series 
representation in \eqref{eq:JS}, while in the large $|\tau|$ 
case it is from moving the line of integration in the integral 
representation in \eqref{eq:JS} to the left, adding the residues at 
$s=1$ and $s=0$. Note that $\zeta(2,b)=b^{-2}+\pi^2/6+O(b)$ and 
$\psi(b)\to b^{-1}$ when $b\to0$. 
\end{proof}

Define 
\begin{align}\label{eq:xi-d-z}
	\varphi_d(z) := \sum_{\ell\ge1}(2\ell-1)^{1-2d}
	\frac{e^{-2(2\ell-1)\pi^2/z}}{1-e^{-2(2\ell-1)\pi^2/z}}
	\qquad(d\in\mathbb{Z}; \Re(z)>0). 
\end{align}

\begin{prop} \label{prop:E01-split}
Uniformly for $|\tau| \to 0$ in the half-plane $\Re(\tau)>0$,
\begin{align}\label{eq:lambda-m-ae}
	\lambda_m\lpa{e^{-4\pi^2/\tau}} 
	= \lpa{1+O\lpa{e^{-\Re(2\pi^2/\tau)}}} \lpa{m^2 \xi_2(m\tau) 
    + \xi_1(m\tau) 
	+ O\lpa{m^{-2}|\xi_0(m\tau)|}},
\end{align}
where
\begin{equation}\label{eq:phi}
    \xi_2(z) 
    := -\frac2{\pi^2}\,\varphi_2(z),
	\quad
	\xi_1(z) 
    := \frac56\,\varphi_1(z),
	\quad
	\xi_0(z) 
    := \varphi_0(z).
\end{equation}
\end{prop}
Note that when $m=O(1)$, the rightmost $O$-term is of the same order 
as $\xi_1(m\tau)\asymp e^{-\Re(2\pi^2/m\tau)}$. 
\begin{proof}
In the defining series \eqref{eq:E01}, we observe that the inner sum
with $z=e^{-4\pi^2/\tau}$ is itself an asymptotic expansion when
$|\tau|\to0$, namely, the term with $k=0$ is dominant and all others
with $k\ge1$ are exponentially smaller. Thus
\begin{align}
	\lambda_m\lpa{e^{-4\pi^2/\tau}}  
	&= - \lpa{1+O\lpa{e^{-\Re(4\pi^2/\tau)}}}  
    	\sum_{1 \le \ell <m} \frac{\cos\lpa{\frac{(2\ell-1)\pi}{m}}}
    	{1-\cos\lpa{\frac{(2\ell-1)\pi}{m}}} 
    	\cdot \frac{e^{-\frac{2\pi^2(2\ell-1)}{m\tau}}}{(2\ell-1) 
    	\lpa{1-e^{-\frac{2\pi^2(2\ell-1)}{m\tau}}}} 
        \nonumber \\
	&= - \lpa{1+O\lpa{e^{-\Re(2\pi^2/\tau)}}}  
    	\sum_{1 \le \ell \le \tr{m/2}} 
        \frac{\cos\lpa{\frac{(2\ell-1)\pi}{m}}}
    	{1-\cos\lpa{\frac{(2\ell-1)\pi}{m}}} 
    	\cdot \frac{e^{-\frac{2\pi^2(2\ell-1)}{m\tau}}}{(2\ell-1) 
    	\lpa{1-e^{-\frac{2\pi^2(2\ell-1)}{m\tau}}}}, 
	\label{eq:E01-refined}
\end{align}
where in the second approximation we truncate terms with 
$\ell>\tr{m/2}$ whose total contribution is bounded above by 
$O\lpa{m^2e^{-\Re(2\pi^2/\tau)}}$.

By expanding the ratio of cosines in \eqref{eq:E01-refined} using the 
inequalities
\[
	-x^2 \le \frac{\cos x}{1-\cos x} 
	- \frac{2}{x^2} 
    + \frac{5}{6} \le x^2
	\qquad(0 \le x \le 1/2),
\]
we then get \eqref{eq:lambda-m-ae} by summing the resulting terms and
extending then the summation range to infinity. The error terms 
introduced are bounded above by 
\begin{align*}
    O\llpa{\sum_{\ell>\tr{m/2}}\Lpa{
    \frac{m^2}{(2\ell-1)^3}+\frac1{2\ell-1}+\frac{2\ell-1}{m^2}}
    e^{-\Re(2(2\ell-1)\pi^2/(m\tau))}}
    =O\lpa{m^{-1}e^{-\Re(2\pi^2/\tau)}}.
\end{align*}
This proves the proposition.
\end{proof}

When $z \to 0$, we see that $\varphi_d(z)$ is itself an asymptotic
expansion. However, when $z \to\infty$, the asymptotic behavior of
$\xi_2, \xi_1, \xi_0$ cannot be read directly from their defining
equations. We now consider this range of $z$. Recall the functions 
$\eta_d(z)$ defined in \eqref{eq:eta-def}, which are themselves 
asymptotic expansions for large $|z|$ in the right half-plane.  

\begin{lem} \label{lmm:phi-large}
The functions $\xi_d(z)$ ($d=0,1,2$) satisfy the identities:
\begin{align}
	\xi_0(z) 
    &= \frac{z^2}{48\pi^2} 
    - \frac1{24} 
	+ \frac{z^2}{2\pi^2}\, \eta_0(z), 
	\label{eq:xi0-large}\\
	\xi_1(z) 
    &= \frac{5z}{96} 
    + \frac{5}{24} 
	\log\Lpa{\frac{\pi}{2z}}
    - \frac{5}{12}\,\eta_1(z), 
	\label{eq:xi1-large} \\
	\xi_2(z) 
	&= -\frac{z}{96} 
    + \frac{7\zeta(3)}{8\pi^2} 
	- \frac{\pi^2}{24z} 
    + \frac{\zeta(3)}{2z^2} 
	- \frac{\eta_2(z)}{z^2}, 
    \label{eq:xi2-large}
\end{align}
which are also asymptotic expansions for large $|z|$ in $\Re(z)>0$.
\end{lem}
\begin{proof}
We apply the same Mellin transform techniques, together with the
functional equation \eqref{eq:zeta-eq} for the Riemann zeta function,
as in the previous section.

Consider first $\xi_2(z)$. By direct calculations using 
\eqref{eq:hurwitz-sum}, we have 
\[
    \xi_2(z) 
    = -\frac2{\pi^2} \cdot\frac1{2\pi i} 
    \int_{(3/2)} X_2(s)z^s \dd s,
\]
where
\[
    X_2(s) = \Gamma(s) \zeta(s) 
    (1-2^{-3-s})\zeta(3+s) (2\pi^2)^{-s}.
\] 
By a similar analysis as in the proof of 
Proposition~\ref{eq:f-expansion}, we deduce that 
\[
    \xi_2(z) 
    = -\frac2{\pi^2}\llpa{\sum_{-2\le k\le 1} \Res_{s=k}(X_2(s)z^s) 
    + \frac1{2\pi i} \int_{(-5/2)} X_2(s)z^s\dd s}.
\]
The sum of the residues yields the first four terms on the right-hand 
side of \eqref{eq:xi2-large}. We then simplify the integral
\[
    \int_{(-5/2)} X_2(s)z^s\dd s
    = \int_{(1/2)} X_2(-2-s)z^{-s-2}\dd s.
\]
By \eqref{eq:zeta-eq},
\begin{align*}
    X_2(-2-s) 
    &= \Gamma(-2-s)\zeta(-2-s)(1-2^{s-1})
      \zeta(1-s)\lpa{2\pi^2}^{s+2} \\
    &= \frac{\pi^2}{2} (1-2^{1-s}) 
      \zeta(s+3)\Gamma(s)\zeta(s),
\end{align*}
which is nothing but the Mellin transform of $\frac{\pi^2}2 
\eta_2(z)$. This proves \eqref{eq:xi2-large}.

The proofs of the other two identities \eqref{eq:xi0-large} and 
\eqref{eq:xi1-large} are similar, and omitted. 
\end{proof}

\begin{coro} Assume $|\tau|\to0$ in the half-plane $\Re(\tau)>0$. 
Then the function $\lambda_m\lpa{e^{-4\pi^2/\tau}}$ satisfies: $(i)$ 
if $m|\tau|\le1$, then 
\begin{equation}\label{eq:lambda-asymp1}
	\lambda_m\lpa{e^{-4\pi^2/\tau}}
    = m^2 \xi_2(m\tau)+\xi_1(m\tau)+
    O\lpa{m^{-2}e^{-\Re(2\pi^2/(m\tau))}};
\end{equation}
$(\text{ii})$ if $m|\tau|\ge1$, then 
\begin{equation}\label{eq:lambda-asymp2}
\begin{split}
	\lambda_m\lpa{e^{-4\pi^2/\tau}}
    &= m^2 \xi_2(m\tau)+\xi_1(m\tau)+O(|\tau|^2)\\
    &= \frac{\zeta(3)-2\eta_2(m\tau)}{2\tau^2}
    -\frac{\pi^2m}{24\tau}+\frac{7\zeta(3)}{8\pi^2}\,m^2
    -\frac{m^3\tau}{96}+\frac{5m\tau}{96}\\
    &\qquad-\frac{5}{24}\log\Lpa{\frac{2m\tau}{\pi}}
    -\frac{5\eta_1(m\tau)}{12}+O(|\tau|^2). 
\end{split}    
\end{equation}
\end{coro}

\section{Asymptotics of $G_{n,m}$}
\label{sec:Gnm}

Our analytic approach to the asymptotics of 
$G_{n,m}$ relies on the Cauchy integral formula
\[
    G_{n,m} 
    = [z^n]G_m(z) 
    = \frac1{2\pi i} \oint_{|z|=e^{-\rho}} z^{-n-1} G_m(z) \dd z
    \qquad(\rho>0).
\]
Since $G_m(e^{-\tau})$ grows very fast near the singularity $\tau=0$
(see \eqref{eq:q-expr-exact}), we will apply the saddle-point method
to the integral on the right-hand side. We derive first crude (but
effective) approximations to $G_{n,m}$ and then sketch our approach
to refining them, more details being given in the next sections.

\subsection{Crude bounds}

By the nonnegativity of the coefficients, we have the simple 
inequality
\begin{align*}
    G_{n,m} 
    &\le e^{n\rho} G_m(e^{-\rho})\\
    &= \exp\Lpa{(n+\phi_m)\rho +\frac{\varpi_m}{\rho}
    +\kappa_m\lpa{e^{-4\pi^2/\rho}}
    +\lambda_m\lpa{e^{-4\pi^2/\rho}}}
    \qquad(n, m\ge 1).
\end{align*}
Here $\rho=\rho(n,m)>0$ is taken to be the saddle-point, namely, it 
satisfies the equation 
\[
    nG_m(e^{-\rho}) = e^{-\rho}G_m'(e^{-\rho}),
    \quad\text{or}\quad
    n + \phi_m = \frac{\varpi_m}{\rho^2}-
    \partial_\rho \Lpa{\kappa_m\lpa{e^{-4\pi^2/\rho}}
    +\lambda_m\lpa{e^{-4\pi^2/\rho}}}.
\]
Consider first the case when $m$ is not too large. More precisely, if 
\[
    \kappa_m\lpa{e^{-4\pi^2/\rho}}
    +\lambda_m\lpa{e^{-4\pi^2/\rho}}
    = O\lpa{m^2e^{-2\pi^2/(m\rho)}}
    = o\Lpa{\frac{\varpi_m}{\rho}} \asymp \frac{m}{\rho},
\]
or, simply $m\rho\to0$, then, by \eqref{eq:kappa-asymp} and
\eqref{eq:lambda-asymp1}, the saddle-point satisfies
\[
    n+\phi_m \sim \frac{\varpi_m}{\rho^2},
    \quad\text{or}\quad
    \rho \sim \sqrt{\frac{\varpi_m}{n+\phi_m}}.
\]
Thus $\rho$ is of order $\sqrt{m/n}$, which in turn specifies the
range of $m$: $m\rho \asymp m^{3/2}/n^{1/2}\to0$, or $m=o(n^{1/3})$. 
In this range of $m$, we see that
\[
	\log G_{n,m}\le 2\sqrt{(n+\phi_m)\varpi_m}(1+o(1))
	\sim \frac{\pi}{\sqrt{6}}\,\sqrt{mn},
\]
which is tight when compared with the asymptotic estimate in 
\eqref{eq:logG-small}. Note that $\kappa_m(e^{-4\pi^2/\rho})$ is not 
uniformly $o(1)$ in this range, although it is of a smaller order 
than $m/\rho$; indeed, if 
\begin{align}\label{eq:m-small-0}
	m\le \frac{6\pi^{2/3}n^{1/3}}
	{(\log n-2\log \log n+\log \omega_n)^{2/3}},
\end{align}
for any sequence $\omega_n$ tending to infinity, then 
\[
	\kappa_m(e^{-4\pi^2/\rho})
	\asymp m^2 e^{-2\pi^2/(m\rho)}
	\asymp \omega_n^{-2/3}\to0.
\]

For larger $m$ with $m\rho\ge \ve>0$, we use \eqref{eq:lambda-m-ae} 
and Lemma~\ref{lmm:phi-large}, giving
\[
    \log G_m(e^{-\rho})
    = \frac{\zeta(3)}{2 \rho^2}+\frac{\pi^2}{24\rho}
    + \frac{\log \rho}{24}
    + O(1),
\]
as $\rho\to0$ and $m\rho\to\infty$. Thus the saddle-point $\rho$ 
satisfies 
\[
	\rho \sim \Lpa{\frac{\zeta(3)}2}^{1/3}n^{-1/3},
\]
implying that 
\[
	\log G_{n,m} \le \frac{3\zeta(3)^{1/3}}{2}\, n^{2/3}
	(1+o(1)),
\]
which is also tight in view of \eqref{eq:logG-m-large}. 

\subsection{The uniform saddle-point approximation}

The tightness of the crude bounds in the previous subsections is
well-known. We now refine these bounds and derive a uniform
asymptotic approximate for $G_{n,m}$.

For convenience, let $\Lambda(z) := \log G_m(z)$ and write the Taylor 
expansion
\begin{align}\label{eq:Lambda}
	\Lambda(e^{-\rho(1+it)})
	= \sum_{k\ge0}\frac{\Lambda_k(\rho)}{k!}(-it)^k,
    \with \Lambda_k(\rho) := 
	\rho^k\sum_{0\le j\le k}\Stirling{k}{j}
	e^{-j\rho}\Lambda^{(j)}(e^{-\rho}),
\end{align}
where $\Stirling{k}{j}$ denotes Stirling numbers of the second kind. 
In particular, 
\[
	\Lambda_1(\rho) = \rho e^{-\rho} \Lambda'(e^{-\rho}),
    \quad\text{and}\quad
	\Lambda_2(\rho) = 
    \rho^2\lpa{e^{-\rho}\Lambda'(e^{-\rho})
	+ e^{-2\rho}\Lambda''(e^{-\rho})}.
\]
As we will see below, each $\Lambda_k(\rho)$ is of the same order as 
$\Lambda(e^{-\rho}) = \log G_m(e^{-\rho})$. 

\begin{thm} \label{thm:Gnm-spa}
Uniformly for $m\ge1$
\begin{align}\label{eq:Gnm-ge}
    G_{n,m} = \frac{\rho e^{n\rho}G_m(e^{-\rho})}{\sqrt{2\pi 
    \Lambda_2(\rho)}}
	\lpa{1+O\lpa{\Lambda_2(\rho)^{-1}}},
\end{align}
where $\rho>0$ solves the equation 
\begin{align}\label{eq:sp-eq}
    n\rho-\Lambda_1(\rho)=0,
    \quad\text{or}\quad
    n = - \partial_\tau \log G_m(e^{-\tau})\bigl|_{\tau=\rho}.
\end{align}
\end{thm}
The extra factor $\rho$ in \eqref{eq:Gnm-ge} is cancelled out with a 
factor $\rho^2$ in $\sqrt{\Lambda_2(\rho)}$. 

We will prove Theorem~\ref{thm:Gnm-spa} in Section~\ref{sec:spa}. The
justification of the finer saddle-point approximation (55) consists
of the following two propositions, which will be proved in Sections
\ref{sec:just} and \ref{sec:local}, respectively.
 
\begin{prop} \label{prop:O}
Let $\delta := (n\rho)^{-2/5}>0$. Then for a certain constant $c'>0$, 
\begin{align}\label{eq:int-O}
    \int_{\delta\rho\le |t|\le\pi}
    e^{n(\rho+it)} G_m(e^{-\rho-it}) \dd t
    =O\lpa{e^{n\rho}G_m(e^{-\rho})e^{-c'(n\rho)^{1/5}}}.
\end{align}
\end{prop}  
\begin{prop} \label{prop:I}
Let $\delta := (n\rho)^{-2/5}>0$. Then, uniformly for 
$|t|\le\delta$, the Taylor expansion \eqref{eq:Lambda} is itself an 
asymptotic expansion as $|t|\to0$.    
\end{prop}

Note that $\delta=(n\rho)^{-2/5}>0$ is a specially tuned parameter,
chosen in the standard way such that $(n\rho)\delta^2\to\infty$ and
$(n\rho)\delta^3\to0$.

\subsection{Justification of the saddle-point method: proof of Proposition~\ref{prop:O}}
\label{sec:just}

Before proving Proposition~\ref{prop:O}, we derive a few useful 
expressions. 

\begin{lem} For $|z|<1$,
\begin{align}\label{eq:log-Gmz}
    G_m(z) = \exp\llpa{\sum_{\ell\ge1}\frac{U_m(z^\ell)}{\ell}},
	\with U_m(z) := \frac{z}{1-z}
    +\frac{z^3(1-z^{m-2})(1-z^{m-1})}
    {(1-z^{2m})(1-z)(1-z^2)}.
\end{align}    
\end{lem}
\begin{proof}
By \eqref{eq:q}, we have, for $|z|<1$,
\begin{align*}
    \log G_m(z) &= -\sum_{k\ge1}\log(1-z^k)
    - \sum_{1\le j<2m}w_m(j)\sum_{k\ge0}\log(1-z^{2mk+j})\\
    &= \sum_{\ell\ge1}\frac{z^\ell}{\ell(1-z^\ell)}
    +\sum_{1\le j<2m}w_m(j)\sum_{\ell\ge1}
    \frac{z^{j\ell}}{\ell(1-z^{2m\ell})}.
\end{align*}
Thus
\[
    U_m(z) = \frac{z}{1-z}
    + \frac{1}{1-z^{2m}}\sum_{1\le j<2m}w_m(j)z^j.
\]
Then \eqref{eq:log-Gmz} follows from \eqref{eq:Wmz}.    
\end{proof}

\begin{lem} For $\rho>0$
\begin{align}\label{eq:log-GmV}
	\frac{|G_m(e^{-\rho+it})|}{G_m(e^{-\rho})}
    \le \exp\lpa{|V_m(e^{-\rho+it})|-V_m(e^{-\rho})}
    \qquad(-\pi\le t\le\pi),
\end{align}
where
\begin{align}\label{eq:Vmz}
    V_m(z) := \frac{z(1-z^{m})}{2(1-z)^2(1+z^m)}.
\end{align}
\end{lem}
\begin{proof}
Since each $U_m(z^\ell)$ contains only nonnegative Taylor 
coefficients, we have, by \eqref{eq:log-Gmz},
\begin{align}\label{eq:log-Gm}
	\frac{|G_m(e^{-\rho+it})|}{G_m(e^{-\rho})}
    \le \exp\lpa{-U_m(e^{-\rho})+\Re(U_m(e^{-\rho+it}))}
    \qquad(-\pi\le t\le\pi).
\end{align}
From \eqref{eq:log-Gmz}, we have the decomposition 
\begin{align}\label{eq:UVmz}
	U_m(z) = V_m(z) + \frac{z^{2m}}{1-z^{2m}}
	+\frac{z}{2(1-z^2)},
\end{align}
where each term contains only nonnegative Taylor coefficients; this 
implies that we also have
\[
	\frac{|G_m(e^{-\rho+it})|}{G_m(e^{-\rho})}
    \le \exp\lpa{-V_m(e^{-\rho})+\Re(V_m(e^{-\rho+it}))},
\]
from which \eqref{eq:log-GmV} follows. 
\end{proof}

Another interesting use of \eqref{eq:log-Gmz} is the following 
very effective way of computing $G_{n,m}$, with only weak dependence 
on $m$. 
\begin{coro} For $m\ge1$, $G_{n,m}$ satisfies $G_{0,m}=1$ and for 
$n\ge1$
\[
    G_{n,m} = \frac1n\sum_{1\le k\le n}G_{n-k,m}
    \sum_{d\divides k} [z^d]zU_m'(z),
\]
where
\begin{align}\label{eq:zd-Umz}
    [z^d]zU_m'(z)
    = \begin{dcases}
        \frac{d}{2}+\frac {dm}4\Lpa{1+(-1)^{\tr{d/m}}
        \Lpa{2\Bigl\{ \frac dm\Bigr\}-1}},
        &\text{if }d \text{ is odd};\\
        \frac{dm}4\Lpa{1+(-1)^{\tr{d/m}}
        \Lpa{2\Bigl\{ \frac dm\Bigr\}-1}},
        &\text{if }d \text{ is even}, d\nmid 2m;\\
        d, &\text{if }d\divides 2m.
    \end{dcases}
\end{align}
\end{coro}
\begin{proof}
Since $(1-x)/(1+x)=1-2x/(1+x)$, we have, by a direct expansion, 
\begin{align}\label{eq:Vmz-coeff}
    V_m(z) 
    &= \frac m4\sum_{d\ge1}
    \Lpa{1+(-1)^{\tr{d/m}}\Lpa{2\Bigl\{ \frac dm\Bigr\}-1}} z^d.
\end{align}
Now taking derivative with respect to $z$ and then multiplying by $z$ 
on both sides of \eqref{eq:log-Gmz} give
\[
    zG_m'(z) = G_m(z) \sum_{\ell\ge1}z^\ell U_m'(z^\ell),
\]
or, taking coefficient of $z^n$ on both sides yields 
\[
    G_{n,m} 
    = \frac1n\sum_{1\le k\le n}G_{n-k,m}
    [z^k]\sum_{\ell\ge1}z^\ell U_m'(z^\ell)
    = \frac1n\sum_{1\le k\le n}G_{n-k,m}
    \sum_{d\divides k} [z^d]zU_m'(z).
\]
By \eqref{eq:UVmz} and \eqref{eq:Vmz-coeff}, we then deduce 
\eqref{eq:zd-Umz}. 
\end{proof}

We now focus on uniform bounds for $|V_m(e^{-\rho-it})|$. 
\begin{prop} For any $3\le m\le n$ and $\rho\to0^+$,
\begin{align}\label{eq:st}
    \frac{|V_m(e^{-\rho-it})|}{V_m(e^{-\rho})}
    \le \begin{dcases}
        1-c\rho^{-2}t^2, & \text{if }|t|\le\rho;\\
        \frac78, & \text{if }\rho\le|t|\le\pi.
    \end{dcases}
\end{align}    
\end{prop}
Before the proof, we observe that $V_m(z)$ admits the partial 
fraction expansion,
\[
    V_m(z) = \frac{m}{4(1-z)}
    +\sum_{1\le j\le m} 
	\frac{e_{m,j}^2}{m(1-e_{m,j})^2
    (e_{m,j}-z)},\with 
    e_{m,j} := e^{(2j+1)\pi i/m},
\]
which shows the subtlety of estimating 
\begin{align}\label{eq:abs-Vm}
    \bigl|V_m(e^{-\rho-it})\bigr|
    &= \frac{e^{-\rho}}{2(1-2e^{-\rho}\cos t+e^{-2\rho})}
    \sqrt{\frac{1-2e^{-m\rho}\cos(mt)+e^{-2m\rho}}
    {1+2e^{-m\rho}\cos(mt)+e^{-2m\rho}}}.
\end{align}

\begin{proof}
Our proof of \eqref{eq:st} is long and divided into several parts.

\paragraph{Growth order of $V_m(e^{-\rho})$.}
By the definition \eqref{eq:Vmz} of $V_m(z)$, we easily obtain the 
estimates
\[
	V_m(e^{-\rho})
	\sim \begin{dcases}
		\frac{m}{4\rho}, &\text{if }m\rho\to0;\\
		\frac{1-e^{-m\rho}}{2\rho^2(1+e^{-m\rho})},
		&\text{if }m\rho \asymp 1;\\
		\frac1{2\rho^2}, 
		&\text{if }m\rho\to\infty. 
	\end{dcases}
\]
In all cases, we have $V_m(e^{-\rho})\asymp n\rho$. 

\paragraph{Uniform bounds for  $|z/(1-z)^2|$.}
We consider first the modulus of $|z/(1-z)^2|$, which is independent 
of $m$ and simpler. Observe that  
\[
    \frac{(1-e^{-\rho})^2}{|1-e^{-\rho-it}|^2}
	= \frac{(1-e^{-\rho})^2}{1-2e^{-\rho}\cos t+e^{-2\rho}}
    = \frac{(1-e^{-\rho})^2}{(1-e^{-\rho})^2 
	+ 2e^{-\rho}(1-\cos t)},
\]
for $-\pi\le t\le \pi$. Now if $|t|=O(\rho)$, then we have the 
uniform expansion
\begin{align}\label{eq:ratio-dn}
	\frac{(1-e^{-\rho})^2}{|1-e^{-\rho-it}|^2}
	= \frac1{1+\rho^{-2}t^2}\Lpa{1+\frac{t^2}{12}
    +\frac{t^2(t^2-\rho^2)}{240}+O(t^6+\rho^4t^2)},
\end{align}
while if $\rho\le |t|\le\pi$, then, by monotonicity,  
\begin{align}\label{eq:ratio-dn2}
	\max_{\rho\le |t|\le\pi}
	\frac{(1-e^{-\rho})^2}{|1-e^{-\rho-it}|^2}
	\le \frac{(1-e^{-\rho})^2}{1-2e^{-\rho}\cos \rho+e^{-2\rho}}
	\sim \frac12.
\end{align}

\paragraph{A uniform bound when $|t|\le\rho$.}

The other factor in \eqref{eq:abs-Vm} is more complicated. For 
convenience, write
\[
    \upsilon(w) := \frac{1-e^{-w}}{2(1+e^{-w})}.
\]
Consider first the range $|t|\le\rho$, beginning with the expression 
\begin{align*}
    \frac{|\upsilon(m(\rho+it))|}{\upsilon(m\rho)}
	&= \sqrt{\frac{1+\frac{2e^{-m\rho}}{(1-e^{-m\rho})^2}
    (1-\cos(mt))}
    {1-\frac{2e^{-m\rho}}{(1+e^{-m\rho})^2}(1-\cos(mt))}}.
\end{align*}
When $|t|\le\rho$, we have the inequality
\begin{equation}\label{eq:ei0}
\begin{split}
    \frac{2e^{-m\rho}}{(1+e^{-m\rho})^2}(1-\cos(mt))
    &\le \begin{dcases}
        \frac{2e^{-m\rho}}{(1+e^{-m\rho})^2}(1-\cos(m\rho)),
        &\text{if }m\rho\le\pi\\
        \frac{4e^{-m\rho}}{(1+e^{-m\rho})^2},
        &\text{if }m\rho>\pi
    \end{dcases} \\
    &< 0.3.
\end{split}   
\end{equation}
Then, by the inequalities 
\[
    \begin{cases}
        (1+x)^{1/2}\le 1+x/2, & \text{for }x\ge 0;\\
        (1-x)^{-1/2}\le 1+2x/3, & \text{for }0\le x\le 0.3,
    \end{cases}
\]
we obtain 
\[
	\frac{|\upsilon(m(\rho+it))|}{\upsilon(m\rho)}
    \le 1+e^{-m\rho}(1-\cos(mt))
	\Lpa{\frac{4}{3(1+e^{-m\rho})^2}
	+\frac1{(1-e^{-m\rho})^2}}
	+ \frac{4e^{-2m\rho} (1-\cos(mt))^2}{3(1-e^{-2m\rho})^2},
\]
and then, by \eqref{eq:ratio-dn}, 
\[
	\frac{|V_m(e^{-\rho-it})|}{V_m(e^{-\rho})}
	\le \frac{1+\Upsilon \rho^{-2}t^2}{1+\rho^{-2}t^2}
	(1+O(t^2)),
\]
where $\Upsilon=\Upsilon(\rho,t)$ is defined as
\begin{align*}
	\Upsilon(\rho,t)
	&:= \rho^2 t^{-2}e^{-m\rho}(1-\cos(mt)) 
	\Lpa{\frac{4}{3(1+e^{-m\rho})^2}
	+\frac1{(1-e^{-m\rho})^2}}\\
	&= \frac{1-\cos(mt)}{(mt)^2/2}
	\cdot e^{-m\rho} \Lpa{\frac{2(m\rho)^2}
	{3(1+e^{-m\rho})^2}
	+\frac{(m\rho)^2}{2(1-e^{-m\rho})^2}}.
\end{align*}
Since $(1-\cos t)/(t^2/2)\le 1$ for all $t\in\mathbb{R}$ and 
\[
	\max_{x\ge0}e^{-x} \Lpa{\frac{2x^2}
	{3(1+e^{-x})^2}+\frac{x^2}{2(1-e^{-x})^2}}
	< 0.65,
\]
we have 
\begin{align}\label{eq:Vm-t-s}
	\frac{|V_m(e^{-\rho-it})|}{V_m(e^{-\rho})}
    \le \frac{1+0.65\rho^{-2}t^2}{1+\rho^{-2}t^2}(1+O(t^2))
	\le 1-c\rho^{-2}t^2,
\end{align}
for $|t|\le \rho$, where $0<c<0.35$. 

\paragraph{A uniform bound when $\rho\le |t|\le\pi$ and $m\rho>\pi$.} 
In this case, we follow the same procedure as above, noting that 
\[
    \frac{2e^{-m\rho}}{(1+e^{-m\rho})^2}(1-\cos(mt))
    \le \frac{4e^{-m\rho}}{(1+e^{-m\rho})^2}<0.19
    < 0.3,
\]
when $m\rho>\pi$ and $|t|\le\pi$. Then 
\[
	\frac{|\upsilon(m(\rho+it))|}{\upsilon(m\rho)}
    \le 1+2e^{-m\rho}\Lpa{\frac{4}{3(1+e^{-m\rho})^2}
	+\frac1{(1-e^{-m\rho})^2}} + \frac{4e^{-2m\rho} (1-\cos(mt))^2}{3(1-e^{-2m\rho})^2}
    <1.25.
\]
This, together with \eqref{eq:ratio-dn2}, gives
\begin{align}\label{eq:Vm-t-b0}
	\frac{|V_m(e^{-\rho-it})|}{V_m(e^{-\rho})}
    < \frac{1.25}{2}=\frac58, 
\end{align}
when $m\rho>\pi$ and $\rho\le |t|\le\pi$. 

\paragraph{A uniform bound when $\rho\le |t|\le\pi$ and
$m\rho\le\pi$.} In this case, $1/(1-z)^2$ has a double pole at $z=1$, 
while $(1-z^m)/(1+z^m)$ has simple poles at $z=e^{t_j i}$ for 
$-\tr{m/2}\le j\le \cl{m/2}$, where $t_j:= (2j-1)\pi/m$. Since 
$1/|1-e^{-\rho-it}|^2$ is monotonically decreasing in $|t|$ when 
$|t|\le\pi$ and $|\upsilon(m(\rho+it))|$ reaches the same maximum at  
$t=t_j$ for all $j$, we then deduce that
\[
	\max_{\rho\le|t|\le\pi}
	|V_m(e^{-\rho-it})|
	\le \max\{|V_m(e^{-\rho-i\rho})|,|V_m(e^{-\rho-it_1})|\},
\]
where $t_1=\pi/m\ge\rho$ when $m\rho\le\pi$. By \eqref{eq:Vm-t-s}, we 
have
\[
	\frac{|V_m(e^{-\rho-i\rho})|}{V_m(e^{-\rho})}
    \le \frac{1.65}{2}(1+O(\rho^2))
	< \frac78.
\]
On the other hand, when $t=t_1$,
\begin{align*}
    \frac{|\upsilon(m(\rho+it_1))|}{\upsilon(m\rho)}
    =\frac{(1+e^{-m\rho})^2}{(1-e^{-m\rho})^2}.
\end{align*}
It follows, by \eqref{eq:ratio-dn}, that 
\begin{align*}
    \frac{|V_m(e^{-\rho-it_1})|}{V_m(e^{-\rho})}
    = \frac{(1+e^{-m\rho})^2}{(1+\pi^2(m\rho)^{-2})
    (1-e^{-m\rho})^2}(1+O(t_1^2))<\frac78,
\end{align*}
when $m\rho\le\pi$, since the value of the monotonic function 
\[
    x\mapsto \frac{(1+e^{-x})^2}{(1+\pi^2x^{-2})
    (1-e^{-x})^2},
\]
lies between $4/\pi^2$ and $0.6$ when $x\in[0,\pi]$. 
Summarizing, we proved that, for $\rho\le|t|\le\pi$, 
\begin{align}\label{eq:Vm-t-b}
	\frac{|V_m(e^{-\rho-it_1})|}{V_m(e^{-\rho})}
	\le \frac78,
\end{align}
whether $m\rho\le\pi$ or $m\rho>\pi$. 

By collecting the estimates \eqref{eq:Vm-t-s}, \eqref{eq:Vm-t-b0},
and \eqref{eq:Vm-t-b}, we obtain \eqref{eq:st} and complete the proof
of the uniform bounds. 
\end{proof}

\begin{proof}(Proposition~\ref{prop:O}: smallness of the 
integral over $\delta\rho\le|t|\le\pi$)
By \eqref{eq:log-GmV}, we obtain
\begin{align*}
	&\int_{\delta\rho\le |t|\le\pi}
	e^{n(\rho+it)} G_m(e^{-\rho-it}) \dd t\\
    &\qquad = O\llpa{e^{n\rho}G_m(e^{-\rho})
    \biggl(\int_{\delta\rho}^\rho+
	\int_\rho^\pi\biggr)
    \exp\lpa{-V_m(e^{-\rho})+\bigl|
    V_m(e^{-\rho-it})\bigr|}\dd t}\\
	&\qquad =: O\lpa{e^{n\rho}G_m(e^{-\rho})
	(J_1+J_2)}.
\end{align*}
By \eqref{eq:Vm-t-s}, for some constants $c, c' > 0$, we have
\begin{align*}
	J_1 = O\llpa{\int_{\delta\rho}^{\rho}
	e^{-cV_m(e^{-\rho})t^2/\rho^2}\dd t}
	= O\lpa{\rho e^{-cV_m(e^{-\rho})\delta^2}}
	= O\lpa{\rho e^{-c'(n\rho)^{1/5}}}.
\end{align*}
On the other hand, by \eqref{eq:Vm-t-b}, $J_2$ is bounded above by 
\begin{align}\label{eq:J2}
	J_2 = O\lpa{e^{-cV_m(e^{-\rho})}}
    = O\lpa{e^{-c'n\rho}}.
\end{align}
This completes the proof of Proposition~\ref{prop:O}.     
\end{proof}

\subsection{Asymptotic nature of the expansion \eqref{eq:Lambda}: proof of Proposition~\ref{prop:I}}
\label{sec:local}

We now prove Proposition~\ref{prop:I} from which the asymptotic
approximation \eqref{eq:Gnm-ge} will then follow.

We begin with the following uniform estimates for $\log
G_m(e^{-\tau})$.
\begin{lem} Let $\tau=\rho+it$. Then, uniformly for $\rho\to0$ and 
	$|t|=O(\rho)$ in the half-plane $\rho>0$,
\begin{align}\label{eq:log-Gm-ue}
	\log G_m(e^{-\tau})
	= \begin{dcases}
		O(m/|\tau|), & \text{if }m\rho\le 1,\\
		O(|\tau|^{-2}), & \text{if }m\rho\ge 1.
	\end{dcases}
\end{align}	
\end{lem}
\begin{proof}
If $m\rho\le 1$, then, by \eqref{eq:kappa-asymp} and 
\eqref{eq:lambda-asymp1}, we obtain
\[
	\kappa_m\lpa{e^{-4\pi^2/\tau}} 
    +\lambda_m\lpa{e^{-4\pi^2/\tau}} 
	= O\lpa{m^2 e^{-\Re(2\pi^2/(m\tau))}}
	= O\lpa{m^2 e^{-c/(m\rho)}},
\]
which is obviously $O(m/|\tau|)$. Now, by \eqref{eq:rm-expr} and the 
asymptotic expansion \eqref{eq:log-gm}, we have
\[
    \log G_m(e^{-\tau}) 
    = O(m/|\tau| + m^2 + m^3|\tau|)
    = O(m/|\tau|),
\]
since $m|\tau|=O(1)$. 
      
On the other hand, if $m\rho\ge1$, then, by \eqref{eq:rm-expr} using 
the expressions in \eqref{eq:f-details}, \eqref{eq:log-gm}, 
\eqref{eq:kappa-asymp} and \eqref{eq:lambda-asymp2}, we deduce that  
\begin{equation}\label{eq:log-G-large-rho}
\begin{split}    
    \log G_m(e^{-\tau})
    &= \frac{\zeta(3)-2\eta_2(m\tau)}{2\tau^2}+\frac{\pi^2}{24\tau}
    +\frac{\log \tau}{24}+\frac{\zeta'(-1)}2-\frac{\log 2}4
    +\frac{\tau}{48}\\
    &\qquad -\frac{5\eta_1(m\tau)}{12} 
    + \frac12p(e^{-m\tau}) + O\lpa{|\tau|^2+m^{-2}},
\end{split}    
\end{equation}
where many terms in $\varpi_m/\tau+\log g_m+\phi_m\tau$ are 
cancelled with the corresponding ones in \eqref{eq:lambda-asymp2}. 
Thus, by \eqref{eq:eta-def}, we have $\log G_m(e^{-\tau}) = 
O(|\tau|^{-2})$.
\end{proof}

\begin{lem} \label{eq:Lk}
For $k\ge0$, we have, uniformly for $|t|=O(\rho)$,
\[
	|\Lambda^{(k)}(e^{-\rho-it})| 
    = O\lpa{\rho^{-k}\Lambda(e^{-\rho})}.
\]	
\end{lem}
\begin{proof}
We apply a standard argument (or Ritt's Lemma; see 
\cite[\S~4.3]{Olver1974}) for the asymptotics of the  derivatives of 
an analytic function in a compact domain, starting from the integral 
representation 
\[
    \Lambda^{(k)}(e^{-\rho-it})
    = \frac{k!}{2\pi i}\oint_{|w-e^{-\rho-it}|=c\rho e^{-\rho}}
    \frac{\Lambda(w)}{(w-e^{-\rho-it})^{k+1}}\,\dd w,
\]
where $c>0$ is a suitably chosen small number. Then, since 
$\rho\to0$, we see that 
\[
    \Lambda^{(k)}(e^{-\rho})
    = O\Lpa{\rho^{-k} \max_{|\theta|\le\pi}
    |\Lambda(e^{-\rho-it}(1+c\rho e^{i\theta}))|}
    = O\Lpa{\rho^{-k} \max_{|\theta|\le\pi}
    |\Lambda(e^{-\rho-it + c\rho e^{i\theta}})|}.
\]
By choosing $c$ sufficiently small, the circular range specified by
$\rho+it -c\rho e^{i\theta}$ for $|\theta|\le\pi$ is covered in the
cone $|t|=O(\rho)$, and we can then apply the bounds for
$\Lambda$ given in \eqref{eq:log-Gm-ue}.
\end{proof}

\begin{proof}(Proposition~\ref{prop:I})
Lemma~\ref{eq:Lk} implies, by the definition \eqref{eq:Lambda}, that 
\[
    \Lambda_k(\rho) \asymp \Lambda(e^{-\rho})
	= \log G_m(e^{-\rho}),
    \qquad(k=1,2,\dots). 
\]
Thus the Taylor expansion \eqref{eq:Lambda} is also an asymptotic 
expansion when $|t|\to0$. 
\end{proof}

\subsection{The saddle-point approximation.}
\label{sec:spa}

Theorem~\ref{thm:Gnm-spa} is a direct consequence of 
Propositions~\ref{prop:O} and \ref{prop:I}. 
 
\begin{proof}(Theorem~\ref{thm:Gnm-spa})
By \eqref{eq:int-O}, we obtain 
\[
    G_{n,m}  = \frac1{2\pi}\int_{-\delta\rho}^{\delta\rho}
    e^{n(\rho+it)} G_m(e^{-\rho-it}) \dd t
    + O\lpa{e^{n\rho}G_m(e^{-\rho})e^{-c'(n\rho)^{1/5}}}.
\]
Then by the expansion \eqref{eq:Lambda}, Proposition~\ref{prop:I} and 
the estimate in Lemma~\ref{eq:Lk}, we have 
\begin{align*}
    &\frac1{2\pi}\int_{-\delta\rho}^{\delta\rho}
    e^{n(\rho+it)} G_m(e^{-\rho-it}) \dd t\\
    &\qquad = \frac{\rho e^{n\rho}G_{m}(e^{-\rho})}{2\pi}
    \int_{-\delta}^{\delta}
    \exp\Lpa{it(n\rho-\Lambda_1(\rho)) -\frac{\Lambda_2(\rho)}{2}\, t^2 + 
    \frac{\Lambda_3(\rho)}{6}(-it)^3 +O(\Lambda(e^{-\rho})t^4)}\dd t.
\end{align*}
Choose $\rho>0$ to be the solution of the equation \eqref{eq:sp-eq},
which exists by the estimates in \eqref{eq:log-Gm-ue}. Then take
$\delta$ as we described above, namely,
$\Lambda_2(\rho)\delta^2\to\infty$ and $\Lambda_2(\rho)\delta^3\to0$.
The evaluation of the integral is then straightforward, and omitted.
\end{proof}

\begin{rmk}
The same calculations lead indeed to an asymptotic expansion of the 
form 
\[
    G_{n,m} \sim \frac{\rho e^{n\rho}G_m(e^{-\rho})}{\sqrt{2\pi 
    \Lambda_2(\rho)}}
	\llpa{1+\sum_{j\ge1} \gamma_j(\rho)
	\Lambda_2(\rho)^{-j}},
\]
for some (messy) coefficients $\gamma_j(\rho)$ depending on $\rho$. 
In particular (for simplicity, $\Lambda_j=\Lambda_j(\rho)$), 
\[
    \gamma_1(\rho) 
    = \frac{3\Lambda_2\Lambda_4-5\Lambda_3^2}
    {24\Lambda_2^{2}}, 
\]
and 
\[	
	\gamma_2(\rho) = 
	\frac{-24\Lambda_2^3\Lambda_6+168\Lambda_2^2
	\Lambda_3\Lambda_5+105\Lambda_2^2\Lambda_4^2-630
	\Lambda_2\Lambda_3^2\Lambda_4+385\Lambda_3^4}
	{1152\Lambda_2^4}.
\]    
\end{rmk}

\section{Phase transitions}
\label{sec:pt}

Based on the less explicit saddle-point approximation
\eqref{eq:Gnm-ge}, we now derive more precise asymptotic estimates
according to the relative growth rate of $m$ with $n^{1/3}$, which prove Theorem~\ref{thm:log-pt}.

\subsection{Subcritical phase: $m=o(n^{1/3}(\log n)^{-2/3})$}

We consider here $m$ in the range 
\begin{align}\label{eq:m-small}
	3\le m\le m_-,\with 
	m_- := \frac{6\pi^{2/3}n^{1/3}}
	{(\log n-\frac12\log \log n+\log \omega_n)^{2/3}},
\end{align}
for any sequence $\omega_n$ tending to infinity; compare 
\eqref{eq:m-small-0}.
\begin{prop} If $m$ lies in \eqref{eq:m-small}, then 
\begin{align}\label{eq:Gnm-small}
	G_{n,m}
    \sim \frac{g_m \sqrt{\varpi_m}}
    {2\sqrt{\pi}\, n}\,e^{2\sqrt{\varpi_m(n+\phi_m)}}
    \sim\frac{g_m\sqrt{\pi m}}
    {4\sqrt{6}\, n}\,e^{2\sqrt{\varpi_m (n+m^3/96)}},
\end{align}
where $g_m, \varpi_m$ and $\phi_m$ are defined in 
\eqref{eq:f-details}. If $m\to\infty$ and still lies in the interval 
\eqref{eq:m-small}, then
\[
	G_{n,m}
    \sim c_1 n^{-1} m^{23/24}
    e^{-c_2m^2+2\sqrt{\varpi_m (n+m^3/96)}},
    \with (c_1,c_2) 
    := \llpa{\frac{e^{\zeta'(-1)/2}\pi^{1/24}}
    {2^{67/24} \sqrt{3}},-\frac{7\zeta(3)}{8\pi^2}}.
\]
\end{prop}
\begin{proof}
When $3\le m\le m_-$, $\log G_m(e^{-\rho})$ satisfies, by 
\eqref{eq:rm-expr} together with the expressions in
\eqref{eq:f-details}, \eqref{eq:kappa-asymp} and
\eqref{eq:lambda-asymp1},
\begin{align}\label{eq:log-Gm-subcritical}
	\log G_m(e^{-\rho})
	= \frac{\varpi_m}{\rho}+\frac12\log\rho +\log g_m
	+ \phi_m\rho + O\lpa{m^2\xi_2(m\rho)},
\end{align}
where $m^2\xi_2(m\rho)\asymp m^2 e^{-2\pi^2/(m\rho)}$, 
and the saddle-point equation has the form (by an argument
similar to the proof of Lemma~\ref{eq:Lk} using \eqref{eq:rm-expr})
\begin{align}\label{eq:sd-pt-eq-sub}
	n+\phi_m
	= \frac{\varpi_m}{\rho^2}-\frac1{2\rho}
	+O\lpa{m^3\xi_2'(m\rho)}.
\end{align}
Asymptotically, we have, by a direct bootstrapping argument,  
\begin{align}\label{eq:rho-s}
	\rho = \sqrt{\frac{\varpi_m}{n+\phi_m}}
	+ O\lpa{n^{-1}+m^{1/2}n^{-1/2} 
    e^{-4\sqrt{6}\pi n^{1/2}/m^{3/2}}}.
\end{align}
Then the upper limit $m_-$ of $m$ in \eqref{eq:m-small} implies that 
the $O$-terms in the above three equations are all of order $o(1)$; 
in particular,
\[
    \left\{
    \begin{split}
        m^3\rho\xi_2'(m\rho)
        &\asymp m\rho^{-1}e^{-2\pi^2/(m\rho)}
        = \Theta(\omega_n^{-2/3}) \to 0,\\
        m^2\xi_2(m\rho)
        &\asymp m^2 e^{-2\pi^2/(m\rho)}
        =o\lpa{m\rho^{-1}e^{-2\pi^2/(m\rho)}}
        =o\lpa{\omega_n^{-2/3}}.
    \end{split}\right.
\]
[This range is slightly smaller than \eqref{eq:m-small-0} because we
need an expansion for $n\rho$ up to $o(1)$ error, or $(n+\phi_m)\rho 
= \varpi_m/\rho-1/2+o(1)$.] Substituting this choice of $\rho$ and 
using \eqref{eq:sd-pt-eq-sub} into \eqref{eq:log-Gm-subcritical}, we 
have
\begin{align*}
	n\rho+\log G_m(e^{-\rho})
	&= \frac{\varpi_m}{\rho}
	+\frac12\log\rho 
	+\log g_m
	+ (n+\phi_m)\rho  +o(1)\\
	&= 2\sqrt{\varpi_m(n+\phi_m)}
	+\frac12\log\rho 
	+\log g_m +o(1).
\end{align*}
On the other hand, we also have 
\[
	\frac{\rho}{\sqrt{2\pi\Lambda_2(\rho)}}
	\sim \frac{\rho^{3/2}}{2\sqrt{\pi \varpi_m}};
\]
thus
\[
	G_{n,m}
	\sim \frac{g_m\rho^2}
	{2\sqrt{\pi \varpi_m}}\,e^{2\sqrt{\varpi_m(n+\phi_m)}},
\]
proving \eqref{eq:Gnm-small} by \eqref{eq:rho-s}. The values of $c_1, c_2$ are computed using \eqref{eq:log-gm}.
\end{proof}

From this estimate, it is straightforward to show that \eqref{eq:hx1}
holds only when $m=o\lpa{n^{1/7}}$:
\begin{align}\label{eq:n-1over7}
    e^{2\sqrt{\varpi_m(n+m^3/96)}}
    = e^{2\sqrt{\varpi_m n} + O(m^{7/2}n^{-1/2})};
\end{align}
and when $n^{1/7}\ll m=o(n^{3/13})$,
\[
    e^{2\sqrt{\varpi_m(n+m^3/96)}}
    = e^{2\sqrt{\varpi_m n} + \sqrt{\varpi_m}\,m^3n^{-1/2}/192
    +O(m^{13/2}n^{-3/2})}.
\]

\paragraph{A connection to the modified Bessel functions.} By the 
same analysis used in the proof of Proposition~\ref{prop:O} (see 
\eqref{eq:J2}), we have  
\[
    G_{n,m} = \frac1{2\pi i}
    \int_{\rho-i\rho}^{\rho+i\rho}
    e^{n\tau}G_m(e^{-\tau}) \dd \tau
    + O\lpa{e^{n\rho} G_m(e^{-\rho}) e^{-c' n\rho}}.
\]
The integral on the right-hand side is indeed well-approximated by
the modified Bessel function when $3\le m\le m_-$ (see 
\eqref{eq:m-small}). By \eqref{eq:q-expr-exact} and
\eqref{eq:lambda-asymp1},
\begin{align*}
    \frac1{2\pi i}\int_{\rho-i\rho}^{\rho+i\rho}
    e^{n\tau}G_m(e^{-\tau}) \dd \tau
    &= \frac{g_m}{2\pi i}\int_{\rho-i\rho}^{\rho+i\rho}
    \sqrt{\tau} e^{(n+\phi_m)\tau + \varpi_m/\tau}
    \lpa{1+O\lpa{m e^{-\Re(2\pi^2/(m\tau))}}}
    \dd \tau \\
    &= \frac{g_m}{2\pi i}\int_{\mathscr{H}}
    \sqrt{\tau} e^{(n+\phi_m)\tau + \varpi_m/\tau}\dd\tau
    +O\lpa{m e^{-\Re(2\pi^2/(m\tau))}+e^{-cn\rho}},
\end{align*}
where $\mathscr{H}$ denotes a Hankel contour, which starts from 
$-\infty$, encircles around the origin counter-clockwise, and then 
returns to $-\infty$ (the exact shape being immaterial). The last 
integral over $\mathscr{H}$ is nothing but the modified Bessel 
function:
\begin{align*}
    G_{n,m}&\sim\frac{g_m}{2\pi i}\int_{\mathscr{H}}
    \sqrt{\tau} e^{(n+\phi_m)\tau + \varpi_m/\tau}\dd\tau\\
     &= g_m\sum_{j\ge0}\frac{\varpi_m^j(n+\phi_m)^{j+3/2}}
     {j!\Gamma(j-1/2)}\\
     &= \frac{g_m(n+\phi_m)^{-3/2}}{4\sqrt{\pi}}
     \Bigl(\lpa{2\sqrt{\varpi_m(n+\phi_m)}-1} 
     e^{2\sqrt{\varpi_m(n+\phi_m)}}
     \\&\phantom{\frac{g_m(n+\phi_m)^{-3/2}}{4\sqrt{\pi}}}\quad
     -\lpa{2\sqrt{\varpi_m(n+\phi_m)}+1} 
     e^{-2\sqrt{\varpi_m(n+\phi_m)}}\Bigr),
\end{align*}
which holds as long as $3\le m\le m_-$. (Numerical fit of the last 
expression is very satisfactory.) 
    
\subsection{Supercritical phase: $m\gg n^{1/3}\log n$}

We now consider $m$ in the following stationary range 
\begin{align}\label{eq:m-big}
    m\ge m_+, \with m_+:= \Lpa{\frac{n}{\zeta(3)}}^{1/3}
	\Lpa{\frac23\log n+\log\log n+\omega_n},
\end{align}
for any sequence $\omega_n$ tending to infinity with $n$.

\begin{prop} If $m_+\le m\le n$, then 
\begin{align}\label{Gnm-large}
    G_{n,m}\sim G_{n,n}\sim c n^{-49/72}
    e^{\beta_1 n^{2/3}+\beta_2 n^{1/3}},
\end{align}
where the constants $(c,\beta_1,\beta_2)$ are defined in 
\eqref{eq:c-b-b}.
\end{prop}
\begin{proof}
For this range of $m$, we have, by \eqref{eq:log-G-large-rho} and the 
definition of $\eta_d$ in \eqref{eq:eta-def},  
\[
\begin{split}    
    \log G_m(e^{-\rho})
    &= \frac{\zeta(3)}{2\rho^2}+\frac{\pi^2}{24\rho}
    +\frac{\log \rho}{24}+\frac{\zeta'(-1)}2-\frac{\log 2}4
    +\frac{\rho}{48}+O\lpa{\rho^{-2}\eta_2(m\rho)
    +e^{-m\rho}+\rho^2},
\end{split}    
\]
and the saddle-point equation 
\begin{align}\label{eq:m-large}
    n+\frac1{48} = \frac{\zeta(3)}{\rho^3}
    +\frac{\pi^2}{24\rho^2}
    -\frac{1}{24\rho} 
    +O\lpa{\partial_\rho(\eta_2(m\rho)/\rho^2)+me^{-m\rho}+\rho}.
\end{align}
Solving asymptotically the saddle-point equation 
\eqref{eq:m-large} gives, with $N := n+\frac1{48}$,
\[
    \rho = \zeta(3)^{1/3}\,N^{-1/3}
    + \frac{\pi^2}{72\zeta(3)^{1/3}}\,N^{-2/3}
    -\frac1{72}\, N^{-1} 
    + O\lpa{n^{-4/3}+
    mn^{1/3} e^{-\zeta(3)^{1/3}m/n^{1/3}}}.
\]
Then we obtain 
\[
    \left\{
    \begin{split}
        \rho\partial_\rho(\eta_2(m\rho)/\rho^2)
        &\asymp m\rho^{-1} e^{-m\rho}
        = O\lpa{e^{-\omega_n}}\to0,\\
        \rho^{-2}\eta_2(m\rho)
        &\asymp \rho^{-2}e^{-m\rho}
        =o\lpa{m\rho^{-1} e^{-m\rho}}
        =o\lpa{e^{-\omega_n}},\\
        n^{1/3}m e^{-\zeta(3)^{1/3}m/n^{1/3}}
        &=O(e^{-\omega_n}).
    \end{split}
    \right.
\]
Thus we have expansions for $n\rho+\log G_m(e^{-\rho})$ and $\rho$ to 
within an error of order $o(1)$, which, together with the relation 
$\Lambda_2(\rho)\sim 3\zeta(3)\rho^{-2}$, gives the same asymptotic 
approximation as in \eqref{eq:m-infty}.    
\end{proof}

\subsection{Critical phase: $\log m\sim \frac13\log n$}

In this range, we begin with the expansion \eqref{eq:log-G-large-rho}
and the approximate saddle-point equation
\begin{equation}\label{eq:m-large-0}
\begin{split}    
    n
    &= \frac{\zeta(3)-2\eta_2(m\rho)
    +m\rho\eta_2'(m\rho)}{\rho^3}
    +\frac{\pi^2}{24\rho^2}
    -\frac{1}{24\rho}
    -\frac1{48} \\
	&\qquad +\frac{5m\eta_1(m\rho)}{12} 
    + \frac{me^{-m\rho} p'(e^{-m\rho})}{2}
    + O(\rho).
\end{split}    
\end{equation}
We recall that, in this regime, $\alpha = m n^{-1/3}$. Define 
\[
    R(\alpha, r)
	:=r^3-\zeta(3)+2\eta_2(\alpha r)
    -\alpha r\eta_2'(\alpha r),
\]
and 
\[
	\sigma(x) := 3\zeta(3)-6\eta_2(x)
	+4x\eta_2'(x)-x^2\eta_2''(x),
\]
where the $\eta_d(x)$ are defined in \eqref{eq:eta-def}. We begin 
with two simple lemmas establishing the positivity of $\sigma$ and 
the existence of a positive solution $r$ of the equation 
$R(\alpha,r)=0$, respectively. 

\begin{lem} The function $\sigma(x)$ is positive for $x>0$.
\end{lem}
\begin{proof}
Note that $\sigma(x) \sim 3 \zeta(3)$ as $x\to\infty$, and $\sigma(x)
\sim \zeta(2)x/2$ as $x\to0$. So the monotonicity of $\sigma(x)$ for
$x\ge 0$ follows from the identity:
\[
	\sigma'(x) = \sum_{j\ge1}
	\frac{e^{-jx}\tilde{\sigma}(jx)}{j^2(1+e^{-jx})^4},
\]
where $\tilde{\sigma}(x) :=
2(1+e^{-x})^2+2(1-e^{-x})x+(1-4e^{-x}+e^{-2x})x^2 >
2+x^2+4e^{-x}(1-x^2) >2.9$ for $x\ge0$.
\end{proof}

Once $m$ is given, $\alpha=m/n^{1/3}$ is fixed and then $r$ can be 
solved from the equation $R(\alpha,r)=0$, which is nothing but  
\eqref{eq:sdpt}. 
\begin{lem}\label{lem:unique-r} For any $\alpha>0$, the equation
$R(\alpha,r)=0$ has a unique solution $r>0$. Moreover, $r =
r(\alpha)$ is increasing as a function in $\alpha$.
\end{lem}
\begin{proof}
Consider the function $\tilde{R}(x) := \zeta(3)-2\eta_2(x)
+x\eta_2'(x)$, which has the explicit series form
\[
	\tilde{R}(x) = \sum_{j\ge1}
	\frac{1-jxe^{-jx}-e^{-2jx}}{j^3(1+e^{-jx})^2}. 
\]
For large $x$, $\tilde{R}(x) \sim \zeta(3)$, while, for small $x$, 
$\tilde{R}(x) \sim \zeta(2)x/4$. Also 
\[
	\tilde{R}'(x) = \sum_{j\ge1}
	\frac{je^{-jx}(1+e^{-jx}+jx(1-e^{-jx}))}
	{j^3(1+e^{-jx})^3}>0,
\]
for $x>0$. Thus for each fixed $\alpha>0$, the equation
$r^3=\tilde{R}(\alpha r)$ has a unique positive solution. 
\end{proof}

We now state the transitional behavior of $G_{n,m}$ for $m\asymp 
n^{1/3}$.
\begin{prop}\label{prop:critical} Let $\alpha = mn^{-1/3}$, where 
$\log m=\frac13(1+o(1))\log n$. Then we have the asymptotic 
approximation
\begin{align}\label{eq:Gnm-critical}
	G_{n,m} = c(\alpha,r) n^{-49/72}
    e^{\beta_1(\alpha,r)n^{2/3}+\beta_2(\alpha,r)n^{1/3}}
    \lpa{1+O\lpa{n^{-1/3}(1+\alpha^{-5/2})}},
\end{align}
uniformly in $m$, where $r$ is the unique positive solution of 
$R(\alpha,r)=0$, $\beta_1(\alpha,r)=G(\alpha)$ in \eqref{eq:logG-ua}: 
\begin{align*}
	\beta_1(\alpha,r) =G(r) = r + 
	\frac{\zeta(3)-2\eta_2(\alpha r)}{2r^2},\quad 
    \beta_2(\alpha,r) := 
    \frac{\pi^2}{24r},
\end{align*}  
and
\[
    c(\alpha,r) 
    :=  \frac{r^{49/24}}{2^{3/4}\sqrt{\pi\sigma(\alpha r)}}
    \exp \left(
        \frac{\zeta'(-1)}2-\frac{5\eta_1(\alpha r)}{12}
        +\frac{p(e^{-\alpha r})}2
        -\frac{\pi^4}{1152\sigma(\alpha r)}
    \right).
\]
\end{prop}
The error term in \eqref{eq:Gnm-critical} suggests that 
\eqref{eq:Gnm-critical} remains valid as long as $m\gg n^{1/5+\ve}$, 
but outside the range $m=\frac13(1+o(1))\log n$ it is simpler to use 
other simpler approximations such as \eqref{eq:Gnm-small} and 
\eqref{Gnm-large}. 
\begin{proof}
Write first $m=\alpha n^{1/3}$ and 
\begin{align}\label{eq:rho-critical}
    \rho = \frac{r}{n^{1/3}}\Lpa{1+\frac{r_1}{n^{1/3}}
    +\frac{r_2}{n^{2/3}}+\cdots },
\end{align}
where the coefficients $r_j=r_j(\rho,\eta_1,\eta_2)$ can be computed 
as follows. Substitute first this expansion into 
\eqref{eq:m-large-0}, expand in decreasing powers of $n$, equate the 
coefficient of each negative power of $n$ on both sides, and then 
solve for $r_1$, $r_2$, \dots, one after another. In this way, we 
obtain, for example,
\begin{equation*}\label{eq:r1-r2}
    \begin{split}
        r_1 &= \frac{\pi^2 r}{24\sigma(\alpha r)},\\
    	r_2 &= \frac{r^2}{\sigma(\alpha r)}
    	\Bigl(-\frac1{24}+\frac5{12}\,\alpha r\eta_1'(\alpha r)
    	+\frac{\alpha r e^{-\alpha r}p'(e^{-\alpha r})}2 \\
    	&\qquad\qquad 
        +\frac{\pi^4}{1152\,\sigma(\alpha r)^2}
        \lpa{2\alpha r \eta_2'(\alpha r) 
        -2(\alpha r)^2 \eta_2''(\alpha r)
        +(\alpha r)^3 \eta_2'''(\alpha r)}\Bigr).
    \end{split}
\end{equation*}
The determination of further terms $r_j$ with $j\ge4$ requires a
longer expansion in \eqref{eq:m-large-0}. The asymptotic estimate
\eqref{eq:Gnm-critical} then follows from substituting the expansion
\eqref{eq:rho-critical} into the uniform saddle-point approximation
\eqref{eq:Gnm-ge} and expand terms up to an error of
$O\lpa{n^{-2/3}}$, together with the relation 
\[
    \Lambda_2(\rho) = \frac{\sigma(\alpha r)}{\rho^2} 
    +\frac{\pi^2 \sigma'(\alpha r)}{24\rho \sigma(\alpha r)}
    + \cdots. 
\]
The more precise error term in \eqref{eq:Gnm-critical} results from
computing more terms in the expansion and examining the asymptotic
behaviors when $\alpha r$ is large and small; we omit the less
interesting details. 
\end{proof}

In particular, the growth of the number of BPPs when their widths get 
close to the typical length behaves asymptotically like a Gumbel 
distribution. 

\begin{coro} \label{coro:typical-length}
Assume that $m$ satisfies 
\begin{equation}\label{eq:typical-length-para}
    \alpha = \frac{m}{n^{1/3}}
    = \frac1{\zeta(3)^{1/3}} 
    \Lpa{\frac{2}{3} \log\Lpa{\frac{n}{\zeta(3)}} +x}.
\end{equation}
Then 
\begin{align}\label{eq:Gnm-Gnn}
    \frac{G_{n,m}}{G_{n,n}} 
    = \exp\Lpa{-e^{-x}\lpa{1+O\lpa{n^{-1/3}\log n}}},
\end{align}
uniformly for $x=o(\log n)$. 
\end{coro}
\begin{proof}
By a standard bootstrapping argument applied to \eqref{eq:sdpt}, we 
have, for large $\alpha$,
$$
    r = \zeta(3)^{1/3}\left(1 - \frac{\zeta(3)^{1/3}\alpha+2}
    {3\zeta(3)}\, e^{-\zeta(3)^{1/3}\alpha}
    \Lpa{1+O\lpa{(1+\alpha^2)e^{-\zeta(3)^{1/3}\alpha}}}\right). 
$$
Along with \eqref{eq:m-infty}, the ratio between $G_{n,m}$ and 
$G_{n,n}$ thus has the form \eqref{eq:Gnm-Gnn}.
\end{proof}

Similar to Theorem~1.1 in \cite{Pittel2005}, we may conclude that
there is an exponential decay of the number of BPPs of size $n$ and
width $m$ when $m$ is close to the typical width, which is of order
$\Theta(n^{1/3}\log n)$. See \cite{Erdos1941} for a similar Gumbel
limiting distribution of the largest part size in random integer
partitions, which is one of the first results of this type, and also
\cite{Mutafchiev2006a} for the same phenomenon in random ordinary
plane partitions.

\section{Phase transitions in $m$-rowed plane partitions}
\label{sec:m-rowed}

Our method of proof extends to some other classes of plane
partitions. For simplicity, we only consider briefly in this section
plane partitions with $m$ rows, which has the known  generating
function (see \cite{Andrews1976})
\[
    \sum_{n\ge0}H_{n,m}z^n
    = \prod_{k\ge1}\lpa{1-z^k}^{-\min\{k,m\}}
    = P(z)^m \tilde{Q}_m(z)
	= \exp\llpa{\sum_{\ell\ge1}\frac{\tilde{U}_m(z^\ell)}{\ell}},
\]
where $H_{n,m}$ denotes the number of $m$-rowed plane 
partitions of $n$, $P$ is given in \eqref{eq:han-xiong}, and 
\[
	\tilde{Q}_m(z) := \prod_{1\le k<m}\lpa{1-z^k}^{m-k}, 
    \quad\text{and}\quad
    \tilde{U}_m(z) := \frac{z(1-z^m)}{(1-z)^2}.
\]
For $2\le m\le 9$, these partitions appear in OEIS with the following 
identities. 
\begin{center}
\begin{tabular}{c|cccc}
$m$ & $2$ & $3$ & $4$ & $5$ \\
OEIS & A000990 & A000991 & A002799 & A001452\\ \hline
$m$ & $6$ & $7$ & $8$ & $9$ \\
OEIS & A225196 & A225197 & A225198 & A225199
\end{tabular}    
\end{center}

For simplicity, we only describe the transitional behavior of $\log 
H_{n,m}$. Define 
\begin{align}\label{eq:eta-t}
    \eta(t) := \sum_{j\ge1}\frac{1-e^{-jt}}{j^3}.
\end{align}
\begin{thm} \label{thm:m-rowed} Let $\alpha := m/n^{1/3}$. Then 
\begin{align}\label{eq:logH-ua}
    \frac{\log H_{n,m}}{n^{2/3}} 
    \sim H(\alpha) 
    := r+r^{-2}\eta(\alpha r),
\end{align}
uniformly as $m\to\infty$ and $m\le n$, where     
$r=r(\alpha)>0$ solves the equation 
\[
    r^3-2\eta(\alpha r)+\alpha r \eta'(\alpha r)= 0.
\]
In particular,
\begin{align}\label{eq:H-alpha-as}
    H(\alpha)
    \sim \begin{dcases}
        \frac{2\pi}{\sqrt{6}}\,\sqrt{\alpha}, 
        & \text{if }\alpha\to0;\\
        3\cdot 2^{-2/3}\zeta(3)^{1/3},
        & \text{if }\alpha\to\infty.
    \end{dcases}
\end{align}
\end{thm}
\begin{proof}(Sketch)
We consider $\tau$ with $\Re(\tau) > 0$. By the Euler-Maclaurin summation formula (see \cite[Chapter~A.7]{Flajolet2009}), we obtain
\begin{align*}
    \log \tilde{Q}_m(e^{-\tau})
    &= \frac{\eta(m\tau)}{\tau^2} 
    +\frac m2\log\Lpa{\frac{2\pi}{\tau}}
    -\frac{\pi^2m}{6\tau}
    -\frac{\log m}{12}+\frac{m\tau}8
    +\zeta'(-1)\\
    &\qquad -\frac1{12}\log\Lpa{\frac{1-e^{-m\tau}}{\tau}}
    -\frac{\tau^2(1+10e^{-m\tau}+e^{-2m\tau})}{2880
    (1-e^{-m\tau})^2}+O\Lpa{\frac{|\tau|^4}{|1-e^{-m\tau}|^4}},
\end{align*}
which holds uniformly as long as $\tau\to0$ and $m\to\infty$.
Then in this range
\begin{align*}
    m\log P(e^{-\tau})+\log \tilde{Q}_m(e^{-\tau})
    &= \frac{\eta(m\tau)}{\tau^2} 
    -\frac{\log m}{12}+\frac{m\tau}{12}
    +\zeta'(-1) -\frac1{12}\log\Lpa{\frac{1-e^{-m\tau}}{\tau}}\\
    &\qquad
    -\frac{\tau^2(1+10e^{-m\tau}+e^{-2m\tau})}{2880
    (1-e^{-m\tau})^2}\\
    &\qquad
    +O\Lpa{\frac{|\tau|^4}{|1-e^{-m\tau}|^4}+me^{-\Re(4\pi^2/\tau)}}.
\end{align*}

In particular, when $m/n^{1/3}\to\infty$, then $\eta(m\tau) \sim 
\zeta(3)$ and $\eta'(m\tau)=o(1)$. Thus $r\sim 
(2\zeta(3))^{1/3}$, and 
\[
    \log\lpa{[z^n]P(z)^m \tilde{Q}_m(z)}
    \sim 3\zeta(3)^{1/3}(n/2)^{2/3},
\]
consistent with \eqref{eq:pp}. On the other hand, when 
$m=o(n^{1/3})$, we use the asymptotic expansion
\[
    \eta(z) = \frac{\pi^2z}{6}+\frac{z^2}4\Lpa{2\log z
    -3}+\sum_{j\ge1}\frac{B_j z^{j+2}}{j\cdot (j+2)!},
\]
the series being convergent when $|z|<2\pi$. Thus in this case, using the saddle point method,
\[
    \log\lpa{[z^n]P(z)^m \tilde{Q}_m(z)}
    \sim \frac{2\pi}{\sqrt{6}}\sqrt{\alpha}n^{2/3}
    =\frac{2\pi}{\sqrt{6}}\sqrt{nm}. 
\]
The theorem is proved by examining the error terms in each case. We 
omit the details. 
\end{proof}

When $m\rho=o(1)$, we can write down more precise expansions, similar 
to \eqref{eq:n-1over7}, beginning with 
\[
    \log \tilde{Q}_m(e^{-\tau})
    \sim \sum_{1\le k<m} (m-k)\log (k\tau)
    + \sum_{j\ge1}\frac{B_j\varsigma_j(m)}{j\cdot j!}\,\tau^j,
\]
while in the case of BPPs the corresponding expansion is a finite one 
(with exponentially smaller error in $1/\tau$). The infinite series 
is divergent when $m|\tau|\ge 2\pi$. Here 
\[
	\varsigma_j(m) := \sum_{1\le k<m} (m-k)k^j
	= \frac{m}{j+1}\lpa{B_{j+1}(m)-B_{j+1}}
	-\frac1{j+2}\lpa{B_{j+2}(m)-B_{j+2}},
\]
is a polynomial in $m$ of degree $j+2$ and divisible by $m(m-1)$, the 
$B_j(x)$ being Bernoulli polynomials; see \eqref{eq:bern-poly}. In 
particular,
\begin{small}
\[
    \;\varsigma_1(m) = \frac{m(m^2-1)}6,\;
    \varsigma_2(m) = \frac{m^2(m^2-1)}{12}.
\]
\end{small}
The saddle-point equation is now of the form 
\begin{align*}
    N := n-\frac{m(2m^2-1)}{24}
    \sim \frac{m\pi^2}{6\rho^2}
    -\frac{m^2}{2\rho}
    -\sum_{j\ge2}\frac{B_j\varsigma_j(m)}{j!}\,\rho^{j-1}.
\end{align*}
Then, writing $\varsigma_j(m) = m\bar{\varsigma}_j(m)$,
\begin{align*}
    \rho &= \sqrt{\frac{m}{n}}
    \llpa{\frac{\pi^2}{6}
    -\frac{m}2\,\rho+\frac{2m^2-1}{24}\,\rho^2
    -\sum_{j\ge2}\frac{B_j\bar{\varsigma}_j(m)}{j!}\,
    \rho^{j+1}}^{1/2} =: r \Psi(\rho),
\end{align*}
where $r := \pi\sqrt{m/(6n)}$ and 
\[
    \Psi(\rho) := \llpa{1-\frac{3m}{\pi^2}\,\rho
    +\frac{2m^2-1}{4\pi^2}\,\rho^2
    -\frac{6}{\pi^2}
    \sum_{j\ge2}\frac{B_j\bar{\varsigma}_j(m)}{j!}\,
    \rho^{j+1}}^{1/2}.
\]
Thus, by the Lagrange Inversion Formula, 
\[
    \rho \sim \sum_{j\ge1}d_j r^j,\with
    d_j = \frac1j[t^{j-1}]\Psi(t)^j.
\]
Since each $d_j=d_j(m)$ is a polynomial in $m$ of degree $m-1$, we 
see that the general term in the expansion of $\rho$ is of the form 
$m^{(3j-2)/2}/n^{j/2}$, which, after substituting such $\rho$ into 
the corresponding saddle-point approximation gives an expansion in 
terms of $r$ as follows:
\begin{align*}
    [z^n]P(z)^m\tilde{Q}_m(z)
    \sim \sqrt{2}\,\pi N^{-(m+5)/4}(m/24)^{(m+3)/4}
    \exp\llpa{\pi\sqrt{\frac{Nm}{6}}
    +\frac{m^2}{4}+\sum_{j\ge1}\frac{e_j(m)}{N^{j/2}}},
\end{align*}
where $e_j(m)$ is a polynomial of degree $(3j+4)/2$. In general, if 
$n^{j_0/(3j_0+4)}\asymp m =o(n^{(j_0+1)/(3j_0+7)})$, we have the 
asymptotic approximation
\[
    [z^n]P(z)^m\tilde{Q}_m(z)
    \sim \sqrt{2}\,\pi N^{-(m+5)/4}(m/24)^{(m+3)/4}
    \exp\llpa{\pi\sqrt{\frac{Nm}{6}}
    +\frac{m^2}{4}+\sum_{1\le j<j_0}\frac{e_j(m)}{N^{j/2}}}.
\]
In particular, if $m=o(N^{1/7})$, then $j_0=0$, while if 
$m=o(N^{1/5})$, then retaining the term $e_1(m)/\sqrt{N}$ and 
dropping the remaining terms yields an error of order $o(1)$.
 
\begin{rmk}($m$-rowed plane partitions whose non-zero parts 
decrease strictly along each row) The generating function now has the 
form (see \cite{Gordon1969})
\begin{align*}
    F_m(z) &:= \prod_{k\ge1}\lpa{1-z^k}^{-\tr{m/2}}
    \times \prod_{k\ge1}\lpa{1-z^{2k-1}}^{-2\{m/2\}}
    \times \prod_{1\le k\le m-2}\lpa{1-z^k}^{\tr{(m-k)/2}} \\
    & = \frac{P(z)^{\tr{m/2}+2\{m/2\}}}{P(z^2)^{2\{m/2\}}} 
    \,\bar{Q}_m(z),
\end{align*}
where $P(z)$ is as in \eqref{eq:han-xiong} and $\bar{Q}_m(z) := 
\prod_{1\le k\le m-2}\lpa{1-z^k}^{\tr{(m-k)/2}}$. Note that
\[
    F_m(z) = \llpa{\frac{P(z)}{P(z^2)}}^{2\{m/2\}}
    \exp\llpa{\sum_{\ell\ge 1}\frac{\bar{U}_m(z^\ell)}{\ell}},
    \with \bar{U}_m(z) 
    := \frac{z^{1+\mathbf{1}_{\text{m odd}}}-z^{m+1}}
    {(1-z)(1-z^2)},
\]
where $\mathbf{1}_{\text{m odd}}$ is the indicator function for $m$ being odd. We then deduce the same type of transitional behavior as that of 
$m$-rowed plane partitions:
\[
    \log\lpa{[z^n]P(z)^m \bar{Q}_m(z)}
    \sim \Lpa{r+\frac{\eta(\alpha r)}{2r^2}}n^{2/3},
\]
where $\eta$ is defined in \eqref{eq:eta-t} and $r>0$ solves the 
equation $2r^3-2\eta(\alpha r)+\alpha r \eta'(\alpha r)= 0$. 
\end{rmk}

\begin{rmk} In a very similar manner, we can derive the phase 
transitions in the asymptotics of 
\[
    [z^n]\prod_{1\le k\le m}\lpa{1-z^k}^{-k},
\]
the difference here being that for small $m=O(1)$ the saddle-point 
method fails and one needs instead the singularity analysis 
\cite{Flajolet2009} for the corresponding asymptotic approximation. 
Indeed, singularity analysis applies when $1\le m=o(n^{1/3})$:
\[
    [z^n]\prod_{1\le k\le m}\lpa{1-z^k}^{-k}
    \sim \frac{[z^n](1-z)^{-m(m+1)/2}}
    {\prod_{1\le k\le m}k^k} 
    \sim \frac{n^{m(m+1)/2-1}}
    {\Gamma(m(m+1)/2)\prod_{1\le k\le m}k^k},
\]
while our saddle-point analysis applies when $m\to\infty$. 
Furthermore, similar to \eqref{eq:logH-ua}, the transitional behavior 
is described by the function 
\[
    \frac{\eta(\alpha r)}{r^2}-\frac{\alpha}r\,
    \text{Li}_2(e^{-\alpha r})
    = \frac1{r^2}\sum_{j\ge1}\llpa{
    \frac{1-e^{-j\alpha r}}{j^3}-\frac{\alpha r e^{-j\alpha r}}{j^2}},
\]
where $\text{Li}_2(z)$ denotes the dilogarithm function, and $r>0$ 
solves the equation
\[
    2\eta(\alpha r)-2\alpha r\text{Li}_2(e^{-\alpha r})
    +(\alpha r)^2\log(1-e^{-\alpha r})=0.
\]
\end{rmk}

\section*{Acknowledgements}
We thank the two referees for their valuable comments and suggestions 
that substantially improved  the paper. 

\bibliographystyle{abbrv}
\bibliography{bpp-col}

\end{document}